%% file: camel.tex
\documentclass[reqno, 12pt]{amsart}
\usepackage{a4wide}
\usepackage{amsmath}
\usepackage{amsfonts}
\usepackage{amsthm}
\usepackage{amssymb}
\usepackage{cite}
\usepackage{graphicx}
\usepackage{graphics}
\usepackage{bm}
\usepackage{dsfont}
\usepackage{hyperref}
\usepackage{xcolor}
\usepackage[font=footnotesize]{caption}
\usepackage[all]{xy}
\numberwithin{equation}{section}
\newtheoremstyle{personal}%
{12pt}%      Space above
{12pt}%      Space below
{\slshape}%         Body font
{}%         Indent amount
{\bfseries}% Theorem head font
{.}%        Punctuation after theorem head
{.5em}%     Space after theorem head
{}%         Theorem head spec (can be left empty, meaning "normal")
\theoremstyle{personal}%
\newtheorem{thm}{Theorem}[section]

\newtheorem{lem}[thm]{Lemma}
\newtheorem{prop}[thm]{Proposition}

\theoremstyle{definition}
\newtheorem{definition}[thm]{Definition}
\newtheorem{rem}[thm]{Remark}

\newcommand{\ud}{\,\mathrm{d}}
\newcommand{\Cont}{\mathrm{Cont}}
\newcommand{\id}{\mathrm{id}}
\newcommand{\crit}{\mathrm{crit}}
\renewcommand{\emptyset}{\varnothing}
\newcommand{\N}{\mathds{N}}
\newcommand{\Z}{\mathds{Z}}
\newcommand{\T}{\mathds{T}}
\newcommand{\R}{\mathds{R}}
\newcommand{\C}{\mathds{C}}
\newcommand{\la}{\left\langle}
\newcommand{\ra}{\right\rangle}

\newcommand{\supp}{\textup{supp}}

\newcommand{\hamc}{\textup{Ham}_{c}}
\newcommand{\contoc}{\textup{Cont}_{0}}

\begin{document}

\title{A contact camel theorem}

\author[S. Allais]{Simon Allais}
\address{Simon Allais,
\'Ecole Normale Sup\'erieure de Lyon,
UMPA\newline\indent  46 all\'ee d'Italie,
69364 Lyon Cedex 07, France}
\email{simon.allais@ens-lyon.fr}
\urladdr{http://perso.ens-lyon.fr/simon.allais/}
\date{August 16, 2018}

\begin{abstract}
We provide a contact analogue of the symplectic camel theorem that holds in $\R^{2n}\times S^1$, and indeed generalize the symplectic camel. Our proof is based on the generating function techniques introduced by Viterbo, extended to the contact case by Bhupal and Sandon, and builds on Viterbo's proof of the symplectic camel.
\end{abstract}

\maketitle

\section{Introduction}

In 1985, Gromov made a tremendous progress in symplectic geometry with his theory of $J$-holomorphic curves \cite{Gro85}. Among the spectacular achievements of this theory, there was his famous non-squeezing theorem: if a round standard symplectic ball $B_r^{2n}$ of radius $r$ can be symplectically embedded into the standard symplectic cylinder $B^{2}_{R}\times\R^{2n-2}$ of radius $R$, then $r\leq R$ (here and elsewhere in the paper, all balls will be open). Other proofs were given later on by the means of other symplectic invariants: see Ekeland and Hofer \cite{EH1, EH2}, Floer, Hofer and Viterbo \cite{FHV}, Hofer and Zehnder \cite{HZ} and Viterbo \cite{Vit}. In 1991, Eliashberg and Gromov discovered a more subtle symplectic rigidity result: the camel theorem \cite[Lemma 3.4.B]{EG}. In order to remind its statement, let us first fix some notation. We denote by $q_1,p_1,\ldots ,q_n,p_n$ the coordinates on $\R^{2n}$, so that its standard symplectic form is given by $\omega=\ud\lambda$, where $\lambda=p_1\,\ud q_1+\cdots + p_n\,\ud q_n = p\ud q$. We consider the hyperplane $P:=\{q_n=0\}\subset\R^{2n}$, and the connected components $P_-:=\{q_n<0\}$ and $P_+:=\{q_n>0\}$ of its complement $\R^{2n}\setminus P$. We will denote by $B_r^{2n}=B_r^{2n}(x)$ the round Euclidean ball of radius $r$ in $\R^{2n}$ centered at some point $x\in\R^{2n}$, and by $P_R:=P\setminus B_r^{2n}(0)$ the hyperplane $P$ with a round hole of radius $R>0$ centered at the origin. The symplectic camel theorem claims that, in any dimension $2n>2$, if there exists a symplectic isotopy $\phi_t$ of $\R^{2n}$  and a ball $B_r^{2n}\subset \R^{2n}$ such that $\phi_0(B_r^{2n})\subset P_-$, $\phi_1(B_r^{2n})\subset P_+$, and $\phi_t(B_r^{2n})\subset \R^{2n}\setminus P_R$ for all $t\in[0,1]$, then $r\leq R$. The purpose of this paper is to prove a contact version of this theorem.

We consider the space $\R^{2n}\times  S^{1}$, where $S^1:=\R/\Z$. We will denote the coordinates on this space by $q_1,p_1,\ldots ,q_n,p_n,z$, and consider the 1-form $\lambda$ defined above also as a 1-form on $\R^{2n}\times S^1$ with a slight abuse of notation. We denote by $\alpha:=\lambda-\ud z$ the standard contact form on $\R^{2n}\times S^1$.
The set of contactomorphisms of $(\R^{2n}\times S^1,\alpha)$ will be denoted by $\Cont(\R^{2n}\times S^1)$ and
the subset of compactly supported contactomorphisms isotopic to the identity will be denoted by $\contoc(\R^{2n}\times S^1)$.
As usual, by a compactly supported contact isotopy of $(\R^{2n}\times S^1,\alpha)$ we will mean a smooth family of contactomorphisms $\phi_t\in\Cont(\R^{2n}\times S^1)$, $t\in[0,1]$, all supported in a same compact subset of $\R^{2n}\times S^1$. In 2006, Eliashberg, Kim and Polterovich \cite{EKP} proved an analogue and a counterpart of Gromov's non-squeezing theorem in this contact setting;  given any positive integer $k\in\N$ and two radii $r,R>0$ such that $\pi r^2\leq k \leq \pi R^2$,
there exists a compactly supported contactomorphism $\phi\in\Cont(\R^{2n}\times S^1)$ such that
$\phi \left(B^{2n}_{R}\times S^1\right)\subset B^{2n}_{r}\times S^1$ if and only if $r=R$; however, if $2n>2$ and  $R<1/\sqrt{\pi}$, then it is always possible to find such a $\phi$.
In 2011, Sandon \cite{San} extended generating function techniques of Viterbo \cite{Vit} and deduced
an alternative proof of the contact non-squeezing theorem.
In 2015, Chiu \cite{Chi} gave a stronger statement for the contact non-squeezing: given any radius $R\geq 1/\sqrt{\pi}$,
there is no compactly supported contactomorphism isotopic to identity $\phi\in\contoc(\R^{2n}\times S^1)$ such that
$\phi(\textup{Closure}(B^{2n}_{R}\times S^1))\subset B^{2n}_{R}\times S^1$.
The same year, an alternative proof of this strong non-squeezing theorem was given by Fraser \cite{Fra16}
(the technical assumption ``$\phi$ is isotopic to identity" is no longer needed in her proof).

Our main result is the following contact analogue of the symplectic camel theorem:

\begin{thm} \label{thm:ccamel}
In dimension $2n+1>3$, if $\pi r^2 < \ell < \pi R^2$ for some positive integer $\ell$
and $B_R^{2n}\times S^{1}\subset P_-\times S^{1}$, 
there is no compactly supported contact isotopy $\phi_t$ of $(\R^{2n}\times S^1,\alpha)$
such that $\phi_0=\id$, $\phi_1(B_R^{2n}\times S^{1})\subset P_+\times S^{1}$, and $\phi_t(B_R^{2n}\times S^{1})\subset (\R^{2n}\setminus P_r)\times S^{1}$ for all $t\in[0,1]$.
\end{thm}

Notice that the squeezing theorem of Eliashberg-Kim-Polterovich implies that Theorem~\ref{thm:ccamel} does not hold if one instead assumes that $\pi R^2<1$.

Theorem \ref{thm:ccamel}
implies the symplectic camel theorem. Indeed, suppose that there exists a symplectic isotopy $\psi_t$ of $\R^{2n}$  and a ball $B_R^{2n}\subset \R^{2n}$ such that $\psi_0(B_R^{2n})\subset P_-$, $\psi_1(B_R^{2n})\subset P_+$, and $\psi_t(B_R^{2n})\subset \R^{2n}\setminus P_r$ for all $t\in[0,1]$, and assume by contradiction that $r<R$. Without loss of generality, we can assume that $\psi_0=\id$ (see \cite[Prop.\ on page 14]{Sik90}) and that the isotopy $\psi_t$ is compactly supported. By conjugating $\psi_t$ with the dilatation $x\mapsto \nu x$, we obtain a new compactly supported  symplectic isotopy $\psi'_t$ with $\psi'_0=\id$ and  a ball $B_{\nu R}^{2n}\subset P_{-}$ such that $\psi_1'(B_{\nu R}^{2n})\subset P_+$, and $\psi_t'(B_{\nu R}^{2n})\subset \R^{2n}\setminus P_{\nu r}$ for all $t\in[0,1]$. If we choose $\nu>0$ large enough, we have $\pi (\nu r)^2 >\ell > \pi (\nu R)^2$ for some $\ell\in\Z$, and the contact lift of $\psi'_t$ to $\R^{2n}\times S^1$ contradicts our Theorem~\ref{thm:ccamel}.

Our proof of the contact camel theorem is based on Viterbo's proof \cite[Sect. 5]{Vit} of the symplectic version, which is given in terms of generating functions. Viterbo's proof is rather short and notoriously difficult to read. For this reason, in this paper we provide a self-contained complete proof of Theorem~\ref{thm:ccamel}, beside quoting a few lemmas from the recent work of Bustillo \cite{Bus}. The generalization of the generating function techniques to the contact setting is largely due to Bhupal \cite{Bhu} and Sandon \cite{San}. In particular, the techniques from \cite{San} are crucial for our work.

\subsection*{Organization of the paper} In Section~\ref{se:pre}, we provide the background on generating functions and the symplectic and contact invariants constructed by means of them.
In Section~\ref{se:prop}, we prove additional properties of symplectic and contact invariants that will be key to the proof of Theorem~\ref{thm:ccamel}.
In Section~\ref{se:thm}, we prove Theorem~\ref{thm:ccamel}.

\subsection*{Acknowledgments} I thank Jaime Bustillo who gave me helpful advice and a better understanding of reduction inequalities.
I am especially grateful to my advisor Marco Mazzucchelli.
He introduced me to generating function techniques and 
gave me a lot of advice and suggestions throughout the writing process.

\section{Preliminaries}\label{se:pre}

In this section, we remind to the reader some known results about generating functions that we will need.

\subsection{Generating functions} \label{se:gf}

Let $B$ be a closed connected manifold.
We will usually write $q\in B$ points of $B$, $(q,p)\in T^{*}B$ for the cotangent coordinates,
$(q,p,z)\in J^{1}B$ for the 1-jet coordinates and $\xi\in\R^{N}$ for vectors of some fiber space.
A generating function on $B$ is a smooth function
$F : B\times \R^{N}\rightarrow\R$ such that $0$ is a regular value of the fiber derivative
$\frac{\partial F}{\partial \xi}$. Then,

\begin{equation*}
\Sigma_{F} := \left\{ (q,\xi)\in B\times\R^{N}\;|\; \frac{\partial F}{\partial \xi}(q;\xi) =0 \right\},
\end{equation*}
is a smooth submanifold called the level set of $F$.

Generating functions give a way of describing Lagrangians and Legendrians of $T^{*}B$ and $J^{1}B$ respectively.
Indeed,
\begin{equation*}
\iota_{F} : \Sigma_{F} \rightarrow T^{*}B,\qquad \iota_{F}(q;\xi) = (q,\partial_{q} F(q;\xi))
\end{equation*}
and
\begin{equation*}
\widehat{\iota}_{F} : \Sigma_{F} \rightarrow J^{1}B,\qquad \widehat{\iota}_{F}(q;\xi) = (q,\partial_{q} F(q;\xi),F(q;\xi))
\end{equation*}
are respectively Lagrangian and Legendrian immersions.
We say that $F$ generates the immersed Lagrangian $L:=\iota_{F}(\Sigma_{F})$
and the immersed Legendrian $L_1:=\widehat{\iota}_{F}(\Sigma_{F})$.
In this paper, we will only consider embedded Lagrangians and Legendrians.

We must restrict ourselves to a special category of generating functions:
\begin{definition}
A function $F:B\times\R^{N}\rightarrow \R$ is quadratic at infinity if
there exists a quadratic form $Q:\R^{N}\rightarrow\R$ such that the differential $\ud F - \ud Q$ is bounded.
$Q$ is unique and called the quadratic form associated to $F$.
\end{definition}
In the following, by generating function we will always implicitly mean generating function quadratic at infinity.
In this setting, there is the following fundamental result:

\begin{thm}[{\cite[Sect. 1.2]{Sik87},\cite[Lemma 1.6]{Vit}}]\label{thm:gfqi}
If $B$ is closed, then any Lagrangian submanifold of $T^{*}B$ Hamiltonian isotopic to the 0-section
has a generating function,
which is unique up to fiber-preserving diffeomorphism and stabilization.
\end{thm}

The existence in this theorem is due to Sikorav,
whereas the uniqueness is due to Viterbo (the reader might also see \cite{The} for the
details of Viterbo's proof). The contact analogous is the following
(with an additional statement we will need later on):

\begin{thm}[{\cite[Theorem 3]{Cha95},\cite[Theorem 3.2]{Che96},\cite[Theorems 25, 26]{ThePHD}}]\label{thm:contgfqi}
If $B$ is closed, then any Legendrian submanifold of $J^{1}B$ contact isotopic to the 0-section has a generating function,
which is unique up to fiber-preserving diffeomorphism and stabilization.
Moreover, if $L_{1}\subset J^{1}B$ has a generating function
and $\phi_{t}$ is a contact isotopy of $J^{1}B$, then there exists a continuous family of generating functions
$F^{t}:B\times\R^{N}\rightarrow\R$ such that each $F^{t}$ generates the corresponding $\phi_{t}(L_1)$.
\end{thm}

\subsection{Min-max critical values}\label{se:minmax}
In the following, $F:B\times \R^{N}\rightarrow \R$ is a smooth function quadratic at infinity of
associated quadratic form $Q$ (generating functions are a special case).
Let $q$ be Morse index of $Q$ (that is the dimension of its maximal negative subspace).
We will denote by $E$ the trivial vector bundle
$B\times\R^{N}$ and, given $\lambda \in\R$, $E^{\lambda}$ the sublevel set $\{ F < \lambda \}\subset E$.

In this paper, $H^{*}$ is the singular cohomology with coefficients in $\R$ and
$1\in H^{*}(B)$ will always denote the standard generator of $H^{0}(B)$ ($B$ is connected).
Let $C>0$ be large enough so that any critical point of $F$ is contained in $\{|F|< C\}$.
A classical Morse theory argument implies that $\left(E^{C},E^{-C}\right)$ is homotopy equivalent
to $B\times (\{Q < C\} ,\{ Q < -C\})$ and the induced isomorphism given by K\"unneth formula:
\begin{equation}\label{eq:kun}
T : H^{p}(B) \xrightarrow{\simeq} H^{p+q}\left(E^{C},E^{-C}\right)
\end{equation}
does not depend on the choice of $C$. So we define
$H^{*}\left(E^{\infty},E^{-\infty}\right) := H^{*}\left(E^{C},E^{-C}\right)$.
We also define
$H^{*}\left(E^{\lambda},E^{-\infty}\right) := H^{*}\left(E^{\lambda},E^{-C}\right)$.

Given any non-zero $\alpha\in H^{*}(B)$, we shall now define its min-max critical value by
\begin{equation*}
c(\alpha, F) := \inf_{\lambda\in\R}\left\{T\alpha \not\in \ker\left(
H^{*}\left(E^{\infty},E^{-\infty}\right) \rightarrow H^{*}\left(E^{\lambda},E^{-\infty}\right)
\right)\right\}.
\end{equation*}
One can show that this quantity is a critical value of $F$ by classical Morse theory.

\begin{prop}[Viterbo \cite{Vit}]\label{prop:minmax}
Let $F_{1}:B\times\R^{N_1}\to\R$ and $F_{2}:B\times\R^{N_2}\to\R$ be generating functions quadratic at infinity normalized so that $F_1(q_0,\xi_0')=F_2(q_0,\xi_0'')=0$ at some pair of critical points $(q_0,\xi_0')\in\crit(F_1)$ and $(q_0,\xi_0'')\in\crit(F_2)$ that project to the same $q_0$. Then:
\begin{enumerate}
\item\label{item:inv} if $F_1$ and $F_2$ generate the same Lagrangian, then $c(\alpha,F_1)=c(\alpha,F_2)$ for all non-zero 
$\alpha\in H^{*}(B)$,
\item\label{item:additivity} if we see the sum $F_{1}+F_{2}$ as a generating function of the form
\begin{equation*}
F_{1}+F_{2} : B\times\R^{N_1+N_2}\to\R,\qquad (F_{1}+F_{2})(q;\xi_{1},\xi_{2}) = F_{1}(q;\xi_{1}) + F_{2}(q;\xi_{2}),
\end{equation*}
then
\begin{equation*}
c(\alpha\smile\beta, F_{1}+F_{2})\geq c(\alpha,F_{1}) + c(\beta,F_{2}),
\end{equation*}
for all $\alpha,\beta\in H^{*}(B)$ whose cup product  $\alpha\smile\beta$ is non-zero.

\item\label{item:duality} if $\mu\in H^{\dim(B)}(B)$ denotes the orientation class of $B$, then
\begin{equation*}
c(\mu,F_1) = -c(1,-F_1).
\end{equation*}
\end{enumerate}
\end{prop}

\begin{proof}
Point (\ref{item:inv}) follows from the uniqueness statement in theorem \ref{thm:gfqi}
See \cite[Prop. 3.3]{Vit} for point (\ref{item:additivity}) and \cite[cor. 2.8]{Vit} for point (\ref{item:duality}).
\end{proof}

When the base space is a product $B=V\times W$, one has the following 

\begin{prop}[{\cite[Prop. 5.1]{Vit}, \cite[Prop. 2.1]{Bus}}]\label{prop:redgf}
Let $F:V\times W\times\R^{N}\rightarrow\R$ be a generating function and let $w\in W$.
Consider the restriction
$F_{w}:V\times\R^{N}\rightarrow\R$, $F_{w}(v;\xi)=F(v,w;\xi)$
(quadratic at infinity on the base space $V$), then
\begin{enumerate}
\item\label{item:ineqred} if $\mu_{2}$ is the orientation class of $W$, then for all non-zero $\alpha\in H^{*}(V)$,
\begin{equation*}
c(\alpha\otimes 1,F)\leq c(\alpha , F_{w})\leq c(\alpha\otimes\mu_{2},F),
\end{equation*}

\item\label{item:eqred} if $F$ does not depend on the $w$-coordinate,
for all non-zero $\alpha\in H^{*}(V)$ and non-zero $\beta\in H^{*}(W)$,
\begin{equation*}
c(\alpha\otimes\beta,F) = c(\alpha , F_{w}).
\end{equation*}
\end{enumerate}
\end{prop}

\subsection{Generating Hamiltonian and contactomorphism}\label{se:gfhc}

Let $\hamc\left(T^{*}M\right)$ be the set of time-1-flows of
time dependent Hamiltonian vector field. Given $\psi\in\hamc\left(T^{*}M\right)$, its graph
$\text{gr}_{\psi} = \text{id}\times\psi : T^{*}M\hookrightarrow T^{*}M\times T^{*}M$ is a Lagrangian embedding
in $\overline{T^{*}M}\times T^{*}M$.
In order to see $\text{gr}_{\psi}(T^{*}M)$ as the 0-section of some cotangent bundle,
let us restrict ourselves to the case $M=\R^{n}\times\T^{k}$.
First we consider the case $k=0$ then we will quotient
$\R^{n+k}$ by $\Z^{k}$ in our construction. Consider the linear symplectic map
\begin{equation*}
\tau : \overline{T^{*}\R^{n}}\times T^{*}\R^{n} \rightarrow  T^{*}\R^{2n},\qquad
\tau(q,p;Q,P) = \left(\frac{q+Q}{2}, \frac{p+P}{2},P-p,q-Q \right)
\end{equation*}
which could also be seen as $(z,Z) \mapsto \left( \frac{z+Z}{2}, J(z-Z) \right)$ where $J$ is the canonical complex
structure of $\R^{2n}\simeq\C^{n}$. The choice of the linear map is not important to deduce results of
Subsections \ref{se:gf} and \ref{se:cap} (in fact, \cite{San}, \cite{Tray} and \cite{Vit} give different choices).
However, we do not know how to show the linear invariance of Subsection \ref{se:sympinv} without this specific choice.

The Lagrangian embedding $\Gamma_{\psi}:= \tau\circ\text{gr}_{\psi}$ defines
a Lagrangian  $\Gamma_{\psi}(T^{*}M)\subset T^{*}\R^{2n}$ isotopic to the zero section through
the compactly supported Hamiltonian isotopy $s\mapsto\tau\circ\text{gr}_{\psi_{s}}\circ\tau^{-1}$ where
$(\psi_{s})$ is the Hamiltonian flow associated to $\psi$.
As $\Gamma_{\psi}(T^{*}M)$ coincides with the 0-section
outside a compact set, one can extend it to a Lagrangian embedding on the cotangent bundle of the compactified space
$L_{\psi}\subset T^{*}\mathds{S}^{2n}$.

In order to properly define $L_{\psi}$ for $\psi\in\hamc\left(T^{*}(\R^{n}\times\T^{k})\right)$,
let $\widetilde{\psi}\in\text{Ham}_{\Z^{k}}(T^{*}(\R^{n+k}))$ be the unique lift of $\psi$ which is also
lifting the flow $(\psi_{s})$ with $\widetilde{\psi}_{0}=\text{id}$. The application $\Gamma_{\widetilde{\psi}}$
gives a well-defined $\Gamma_{\psi}: T^{*}(\R^{n}\times\T^{k})\hookrightarrow T^{*}(\R^{2n}\times\T^{k}\times\R^{k})$.
We can then compactify the base space: 
$\R^{2n}\times\T^{k}\times\R^{k}\subset B$
where $B$ equals either $\mathds{S}^{2n}\times\T^{k}\times\mathds{S}^{k}$
or $\mathds{S}^{2n}\times\T^{2k}$
and define $L_{\psi}\subset T^{*}B$.

In order to define $F_{\psi}:B\times\R^{N}\rightarrow\R$,
take any generating function of $L_{\psi}$ normalized such that the set of critical points outside
$(\R^{2n}\times\T^{k}\times\R^{k})\times\R^{N}$ has critical value 0
(the set is connected since $L_{\psi}$ coincides with 0-section outside $T^{*}(\R^{2n}\times\T^{k}\times\R^{k})$).

We now extend the construction of $L_{\psi}$ to the case of contactomorphisms. 
Let $\contoc(J^{1}M)$ be the set of contactomorphisms isotopic to identity through compactly supported contactomorphisms.
Given any $\psi\in\hamc\left(T^{*}M\right)$, its lift 
\begin{equation*}
\widehat{\psi}:J^1 M\rightarrow J^1 M,\qquad \widehat{\psi}(x,z) = \left(\psi(x),z+a_{\psi}(x)\right)
\end{equation*}
belongs to $\contoc(J^{1}M)$,
where $a_{\psi}:T^{*}M\rightarrow\R$ is the compactly supported function satisfying
\begin{equation*}
\psi^{*}\lambda - \lambda = \ud a_{\psi}.
\end{equation*}
In \cite{Bhu}, Bhupal gives a mean to define a generating function $F_{\phi}$
associated to such contactomorphism $\phi$ for
$M=\R^{n}\times\T^{k}$ in a way which is compatible with $\psi\mapsto\widehat{\psi}$ in the sense that
$F_{\widehat{\psi}}(q,z;\xi)=F_{\psi}(q;\xi)$. Given any $\phi\in\contoc(J^1\R^{n})$ with 
$\phi^{*}(\ud z-\lambda) = e^{\theta}(\ud z-\lambda) $,
\begin{equation*}
\widehat{\textup{gr}_{\phi}} :
J^{1}\R^{n} \rightarrow J^{1}\R^{n}\times J^{1}\R^{n}\times \R,\qquad
\widehat{\textup{gr}_{\phi}}(x) = (x,\phi(x),\theta(x))
\end{equation*}
is a Legendrian embedding if we endow $J^{1}\R^{n}\times J^{1}\R^{n}\times \R$ with the contact structure
$\ker(e^{\theta}(\ud z-\lambda)-(\ud Z-\Lambda))$, where $(q,p,z;Q,P,Z;\theta)$ denotes coordinates
on $J^{1}\R^{n}\times J^{1}\R^{n}\times \R$ and $\Lambda = \sum_{i} P_{i}\ud Q_{i}$.

For our choice of $\tau$, we must take the following contact identification
\begin{equation*}
\widehat{\tau} :
J^{1}\R^{n}\times J^{1}\R^{n}\times \R \rightarrow  J^{1}\R^{2n+1},
\end{equation*}
\begin{equation*}
\widehat{\tau}(q,p,z;Q,P,Z;\theta) =
\left( \frac{q+Q}{2},\frac{e^{\theta}p+P}{2},z;P-e^{\theta}p,q-Q,e^{\theta}-1;\frac{1}{2}(e^{\theta}p+P)(q-Q)+Z-z\right)
\end{equation*}
so that $\Gamma_{\phi}:=\widehat{\tau}\circ\widehat{\textup{gr}_{\phi}}$ is an embedding of
a Legendrian compactly isotopic to
the 0-section of $J^{1}\R^{2n+1}$. The construction of $\Gamma_{\phi}$ descends well from $\R^{n+k}$ to $\R^{n}\times\T^{k}$
taking the lift of $\phi\in\contoc(J^{1}(\R^{n}\times\T^{k}))$ which is contact-isotopic to identity.

In fact we will rather be interested by $T^{*}M\times S^{1}\simeq J^{1}M/\Z\frac{\partial}{\partial z}$ and
$\phi\in\contoc(T^{*}M\times S^{1})$
which can be identified to the set of $\Z\frac{\partial}{\partial z}$-equivariant
contactomorphism of $J^{1}M$ isotopic to identity. The construction descends well to the last quotient and we obtain
a well-defined Legendrian embedding
$\Gamma_{\phi}: T^{*}(\R^n\times \T^{k})\times S^{1}\hookrightarrow 
J^{1}\left(\R^{2n}\times\T^{k}\times\R^{k}\times S^{1}\right)$.

We can then compactify the base space
$\R^{2n}\times\T^{k}\times\R^{k}\times S^{1}\subset
B\times S^{1}$,
define $L_{\psi}\subset J^1(B\times S^{1})$
and take as $F_{\phi}$ any generating function of $L_{\phi}$ normalized such that the set of critical points outside
$(\R^{2n}\times\T^{k}\times\R^{k}\times S^{1})\times\R^{N}$ has critical value 0.

\subsection{Symplectic and contact invariants}\label{se:cap}
The symplectic invariants presented here are due to Viterbo \cite{Vit}.
The generalization to the contact case is due to Sandon \cite{San}.

throughout this subsection, $B$ denotes a compactification of $T^{*}(\R^n\times\T^k)$.
Given any $\psi\in\hamc\left(T^{*}(\R^{n}\times\T^{k})\right)$
and any non-zero $\alpha\in H^{*}(B)$, consider
\begin{equation*}
c(\alpha, \psi) := c\left(\alpha, F_{\psi}\right).
\end{equation*}

\begin{prop}[{Viterbo, \cite[Prop. 4.2, Cor. 4.3, Prop. 4.6]{Vit}}]\label{prop:cham}
Let $(\psi_t)$ be a compactly supported Hamiltonian isotopy of $T^{*}(\R^n\times\T^k)$ with
$\psi_0=\id$ and $\psi:=\psi_1$. Let $H_t:T^{*}(\R^n\times\T^k)\rightarrow\R$ be
the Hamiltonians generating $(\psi_t)$.
Given any non-zero $\alpha\in H^{*}(B)$,
\begin{enumerate}
\item\label{item:chamA}
There is a one-to-one correspondence between critical
points of $F$ and fixed points $x$ of $\psi$ such that $t\mapsto \psi_{t}(x)$
is a contractible loop when $t\in[0,1]$
given by $(x,\xi)\mapsto x$. Moreover, if
$(x_\alpha,\xi_\alpha)\in\crit(F_\psi)$ satisfies $F_\psi(x_\alpha,\xi_\alpha)=c(\alpha,F_\psi)$, then
\begin{equation*}
c(\alpha,\psi) = a_{\psi}(x_{\alpha})
= \int_{0}^{1}\left(\la p(t),\dot{q}(t) \ra - H_{t}(\psi_t(x_{\alpha}))\right)\ud t,
\end{equation*}
where $(q(t),p(t)):=\psi_t(x_{\alpha})$. The value $a_{\psi}(x)$ will be called the \emph{action} of the fixed point $x$.

\item\label{item:chamH}
If $H_{t}\leq 0$, then $c(\alpha,\psi)\geq 0$.

\item\label{item:chaminv}
If $(\varphi^{s})$ is a symplectic isotopy of $T^{*}(\R^{n}\times\T^{k})$, then
$s\mapsto c(\alpha,\varphi^{s}\circ\psi\circ(\varphi^{s})^{-1})$ is constant.

\item\label{item:chamid}
If $\mu$ is the orientation class of $B$,
\begin{equation*}
c(1, \psi) \leq 0 \leq c(\mu, \psi) \quad \text{with} \quad
c(1, \psi) = c(\mu, \psi)\; \Leftrightarrow\; \psi = \text{id},
\end{equation*}
\begin{equation*}
c(\mu,\psi) = -c(1,\psi^{-1}).
\end{equation*}

\end{enumerate}
\end{prop}
These results were not stated with this generality in \cite{Vit} but the proofs given by Viterbo
immediately generalize to this setting.

Given any open bounded subset $U\subset T^{*}(\R^{n}\times\T^{k})$ and any non-zero
$\alpha\in H^{*}(B)$,
Viterbo defines the symplectic invariant
\begin{equation*}
c(\alpha, U) := \sup_{\psi\in \hamc(U)} c(\alpha, \psi).
\end{equation*}
This symplectic invariant extends to any unbounded open set $U\subset T^{*}(\R^{n}\times\T^{k})$
by taking the supremum of the $c(\alpha,V)$ among
the open bounded subsets $V\subset U$.

\begin{prop}[Bustillo, Viterbo]\label{prop:cU}
For all open bounded sets $U,V\subset T^{*}(\R^{n}\times\T^{k})$ and any non-zero
$\alpha\in H^{*}(B)$,
\begin{enumerate}
\item\label{item:cUinv}
if $(\varphi^{s})$ is a symplectic isotopy of $T^{*}(\R^n\times\T^k)$, then
$s\mapsto c(\alpha, \varphi^{s}(U))$ is constant,
\item\label{item:cUmon}
$U\subset V$ implies $c(\alpha, U)\leq c(\alpha,V)$,
\item\label{item:cUred}
if $\mu_{1}$ and $\mu_2$ are the orientation classes of the compactification of $T^{*}(\R^{n}\times\T^{k})$
and $\R^k$ respectively, then for any neighborhood $W$ of $0\in\R^{k}$,
\begin{equation*}
c(\mu_1,U)\leqslant c(\mu_1\otimes\mu_2\otimes 1, U\times W\times \T^k).
\end{equation*}
\item\label{item:cUB}
if $B^{2n+2k}_{r}\subset T^{*}(\R^{n}\times\T^{k})$ is an embedded round ball of radius $r$
and $\mu$ is the orientation class of
$B$, then
$c(\mu,B^{2n+2k}_{r}) = \pi r^2.$
\end{enumerate}
\end{prop}

\begin{proof}
Point (\ref{item:cUinv}) is a consequence of Proposition~\ref{prop:cham} (\ref{item:chaminv}).
Point (\ref{item:cUmon}) is a consequence of the definition as a supremum.
Point (\ref{item:cUred}) is proved in the proof of \cite[Prop. 2.3]{Bus}.
Indeed, Bustillo makes use of (\ref{item:cUred}) to deduce his Proposition 2.3
by taking the infimum of $c(\mu_1\otimes\mu_2\otimes 1, U\times V\times \T^k)$
among neighborhoods $U\supset X$ and $W\supset \{ 0\}$ (using Bustillo's notations).
We refer to \cite[Sect. 3.8]{ABKLB} for a complete proof of (\ref{item:cUB}).
\end{proof}

Now, we give the contact extension of these invariants.
Given any $\phi\in\contoc\left(T^{*}(\R^{n}\times\T^{k})\times S^1\right)$
and any non-zero $\alpha\in H^{*}(B\times S^1)$, consider
\begin{equation*}
c(\alpha, \phi) := c\left(\alpha, F_{\phi}\right).
\end{equation*}
The following Proposition is due to Sandon. Since our setting is slightly different,
we provide precise references for the reader's convenience.

\begin{prop}[{Sandon, \cite{San}}]\label{prop:ccont}
\begin{enumerate}
\item\label{item:ccontdual}
Given any $\phi\in\contoc(T^{*}(\R^{n}\times\T^{k})\times S^1)$,
if $\mu$ is the orientation class of $B\times S^1$, then
\begin{equation*}
c(\mu,\phi) = 0\quad \Leftrightarrow\quad  c(1,\phi^{-1}) = 0.
\end{equation*}
\item\label{item:ccontminus}
Given any $\phi\in\contoc(T^{*}(\R^{n}\times\T^{k})\times S^1)$ and
any non-zero $\alpha\in H^{*}(B\times S^1)$,
if $F_{\phi}$ is a generating function of $\phi$, then
\begin{equation*}
\left\lceil c\left(\alpha,\phi^{-1}\right)\right\rceil = \left\lceil c\left(\alpha,-F_{\phi}\right)\right\rceil.
\end{equation*}

\item\label{item:ccontinv}
Given any $\phi\in\contoc(T^{*}(\R^{n}\times\T^{k})\times S^1)$,
any non-zero $\alpha\in H^{*}(B\times S^1)$,
if $(\psi^s)$ is a contact isotopy of $T^{*}(\R^{n}\times\T^{k})\times S^1$, then
$s\mapsto \lceil c(\alpha,\psi^{s}\circ\phi\circ(\psi^{s})^{-1})\rceil$
is constant.

\item\label{item:ccontcomp}
For each $\psi\in\hamc(T^{*}(\R^{n}\times\T^{k}))$ for each non-zero cohomology class
$\alpha\in H^{*}(B)$,
if $\ud z$ denotes the orientation class of $S^1$,
then
\begin{equation*}
c\left(\alpha\otimes 1,\widehat{\psi}\right) =
c\left(\alpha\otimes\ud z,\widehat{\psi}\right)=c\left(\alpha,\psi\right).
\end{equation*}
\end{enumerate}
\end{prop}

\begin{proof}
Let $F_{\phi}$ be the generating function of
$\phi\in\contoc(T^{*}(\R^{n}\times\T^{k})\times S^1)$.
According to duality formula in Proposition~\ref{prop:minmax} (\ref{item:duality}), $c(\mu,\phi) = -c(1,-F_{\phi})$.
Points (\ref{item:ccontdual}) and (\ref{item:ccontminus}) then follow from
\cite[lemmas 3.9 and 3.10]{San}
taking $L=0\text{-section}$ and $\Psi = \widehat{\tau}\circ\widehat{\textup{gr}_{\phi^{-1}}}\circ\widehat{\tau}^{-1}$:
\begin{equation*}
c(1,F_{\phi^{-1}}) = 0\quad \Leftrightarrow\quad c(1,-F_{\phi}) = 0
\end{equation*}
and
\begin{equation*}
\left\lceil c\left(1,F_{\phi^{-1}}\right)\right\rceil = \left\lceil c\left(1,-F_{\phi}\right)\right\rceil.
\end{equation*}
Point (\ref{item:ccontinv}) is a consequence of \cite[lemma 3.15]{San} applied to
$c_{t}=c(\alpha,\psi^{t}\circ\phi\circ(\psi^{t})^{-1})$.
Point (\ref{item:ccontcomp}) is given by the proof of \cite[Prop. 3.18]{San}. Indeed,
let $i_{a}:(E^a,E^{-\infty})\hookrightarrow (E^\infty,E^{-\infty})$ and
$\widetilde{i_{a}}:(\widetilde{E}^a,\widetilde{E}^{-\infty})\hookrightarrow
(\widetilde{E}^\infty,\widetilde{E}^{-\infty})$ be the inclusion maps of sublevel sets of $F_{\psi}$ and
$F_{\widehat{\psi}}$ respectively.
Then $\widetilde{E}^a = E^a\times S^1$ and, after identifying
$H^{*}(\widetilde{E}^a,\widetilde{E}^{-\infty})$
with $H^{*}(E^a,E^{-\infty})\otimes H^{*}(S^1)$, the induced maps in cohomology
$\widetilde{i_{a}}^{*}$
is given by
\begin{equation*}
\widetilde{i_{a}}^{*} = i_{a}^{*}\otimes\textup{id}.
\end{equation*}
Thus $\widetilde{i_{a}}^{*}(\alpha\otimes\beta) = (i_{a}^{*}\alpha)\otimes\beta$ is
non-zero if and only if $i_{a}^{*}\alpha$ is non-zero, where $\beta\in\{\ud z, 1\}$.
\end{proof}

Let $\mu$ be the orientation class of $B\times S^1$,
given any open bounded subset $U\subset T^{*}(\R^{n}\times\T^{k})\times S^{1}$ and any non-zero
$\alpha\in H^{*}(B\times S^1)$, consider
\begin{equation*}
c(\alpha, U) := \sup_{\phi\in \contoc(U)} \lceil c(\alpha, \phi)\rceil
\end{equation*}
and
\begin{equation*}
\gamma (U) := \inf\left\{ \left\lceil c\big(\mu,\phi\big)\right\rceil +
\left\lceil c\left(\mu,\phi^{-1}\right)\right\rceil
|\ \phi\in\contoc\left(T^{*}(\R^{n}\times\T^{k})\times S^{1}\right)\text{ such that } \phi(U)\cap U = \emptyset \right\},
\end{equation*}
These contact invariants extend to any unbounded open set $U\subset T^{*}(\R^{n}\times\T^{k})\times S^{1}$
by taking the supremum among
the open bounded subsets $V\subset U$.

\begin{prop}[Sandon \cite{San}]\label{prop:ccU}
For all open bounded sets $U,V\subset T^{*}(\R^{n}\times\T^{k})\times S^1$ and any non-zero
$\alpha\in H^{*}(B\times S^1)$,
\begin{enumerate}
\item\label{item:ccUinv}
if $(\varphi^{s})$ is a contact isotopy of $T^{*}(\R^{n}\times\T^{k})\times S^1$, then
$s\mapsto c(\alpha, \varphi^{s}(U))$ is constant,
\item\label{item:ccUmon}
$U\subset V$ implies $c(\alpha, U)\leq c(\alpha,V)$,
\item\label{item:ccUcomp}
given any open subset
$W\subset T^{*}(\R^{n}\times\T^{k})$, for each non-zero class
$\beta\in H^{*}(B)$,
if $\ud z$ denotes the orientation class of $S^1$, then
\begin{equation*}
c(\beta\otimes\ud z,W\times S^{1}) = \lceil c(\beta, W)\rceil.
\end{equation*}
\end{enumerate}
\end{prop}

\begin{proof}
Point (\ref{item:ccUinv}) is a direct consequence of Proposition~\ref{prop:ccont} (\ref{item:ccontinv}).
Point (\ref{item:ccUmon}) is a consequence of the definition as a supremum.
Point (\ref{item:ccUcomp}) follows from the proof of \cite[Prop. 3.20]{San}. Indeed,
inequality $c(\beta\otimes\ud z,W\times S^{1}) \geq \lceil c(\beta, W)\rceil$
is due to Proposition~\ref{prop:ccont} (\ref{item:ccontcomp}) whereas the
other one is due to the fact that, for all $\phi\in\contoc (W\times S^1)$,
one can find $\psi\in\hamc(W)$ such that $\phi\leq \widehat{\psi}$
in Sandon's notations (see her proof for more details).
\end{proof}

\section{Some properties of symplectic and contact invariants}\label{se:prop}

\subsection{Estimation of $\gamma \left(T^{*}C\times B_{R}^{2n-2}\times S^{1}\right)$}

Here, we will prove the following

\begin{lem}\label{lemma:contgamma}
Let $R>0$ be such that $\pi R^2\not\in\Z$, $b>0$, $n>1$ and $C:=\R/d\Z$. Then
\begin{equation*}
\gamma \left(T^{*}C\times B_{R}^{2n-2}\times S^{1}\right) \leq
\left\lceil \pi R^2 \right\rceil.
\end{equation*}
\end{lem}

\begin{rem}
This Lemma fails for $n=1$. The use of Lemma \ref{lemma:contgamma} will be the step where
we will need the assumption that $2n+1>3$ in the proof of Theorem~\ref{thm:ccamel}.
\end{rem}

In order to prove Lemma~\ref{lemma:contgamma}, we will need the following elementary fact:

\begin{figure}
\begin{small}
\def\svgwidth{0.3\textwidth}
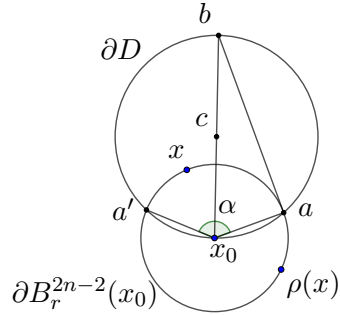
\end{small}
\caption{Configuration in the plane $P$}
\label{fig:elementary}
\end{figure}

\begin{lem}
Let $x_0\in\partial B_{R}^{2n-2}$, $x\in B_{R}^{2n-2}$ and $r:=|x-x_0|$.
We set $\theta(r)\in [0,\pi]$ to be such that
\begin{equation*}
\cos\left(\frac{\theta(r)}{2}\right) = \frac{r}{2R}
\end{equation*}
Then any rotation $\rho:\R^{2n-2}\rightarrow\R^{2n-2}$ of angle $\theta(r)$ centered at $x_0$
sends $x$ outside $B_{R}^{2n-2}$, \emph{i.e.}
$\rho(x)\not\in B_{R}^{2n-2}$.
\end{lem}

\begin{proof}
Let $x_0\in\partial B_{R}^{2n-2}$, $x\in B_{R}^{2n-2}$ and $r:=|x-x_0|$.
Take any rotation $\rho:\R^{2n-2}\rightarrow\R^{2n-2}$ of angle $\theta(r)$, where $\theta(r)$ is
defined as above.
Let $P\subset\R^{2n-2}$ be the affine plane spanned by $x_0$, $x$ and $\rho(x)$.
The round disk $B_{R}^{2n-2}\cap P$ has a radius smaller than $R$ and lies in an open round disk $D$ of radius $R$
with $x_0\in\partial D$ centered at $c\in P$. Therefore, it is enough to show that $\rho(x)\not\in D$.
Let $a,a'$ be the two points of $\partial D\cap \partial B_{r}^{2n-2}(x_0)$, $b$
be the second point of $\partial D\cap (x_0 c)$ and $\alpha$ be the unoriented angle $\widehat{ax_0 a'}\in [0,\pi]$
(see Figure~\ref{fig:elementary}).
As $[x_0 b]$ is a diameter of $\partial D$, the triangle $abx_0$ is right at $a$,
thus $\frac{\alpha}{2}=\widehat{ax_0 b}$ satisfies
\begin{equation*}
\cos \left(\frac{\alpha}{2}\right)= \frac{ax_0}{bx_0} = \frac{r}{2R}. 
\end{equation*}
Hence, $\alpha=\theta(r)$ and $\rho(x)\not\in D$.
\end{proof}

\begin{proof}[Proof of Lemma~\ref{lemma:contgamma}]
\begin{figure}
\begin{small}
\def\svgwidth{0.7\textwidth}
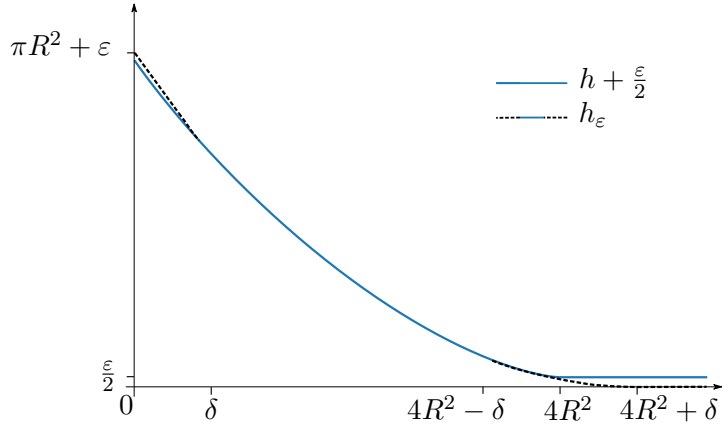
\end{small}
\caption{Approximating $h+\frac{\varepsilon}{2}$ by a smooth compactly supported $h_\varepsilon$.}
\label{fig:function}
\end{figure}

We exhibit a family of $\psi^{\varepsilon}\in\hamc(\R^{2n-2})$ satisfying
\begin{itemize}
\item $\psi^{\varepsilon}\left(B_{R}^{2n-2}\right)\cap B_{R}^{2n-2} = \emptyset$,
\item $c(1,\psi^{\varepsilon})=0$,
\item $\forall\varepsilon >0$, $c(\mu,\psi^{\varepsilon})\leq \pi R^2 + \varepsilon$,
\end{itemize}
where $\mu$ is the orientation class of the compactified space $\mathds{S}^{2n-2}$.
Consider the radial Hamiltonian $H(x)=-h(r^{2})$ where $h:[0,+\infty)\rightarrow \R$ is defined by:
\begin{equation*}
h(u) =\frac{1}{2}\int_{0}^{4R^{2}} \theta(\sqrt{v})\ud v
-\frac{1}{2}\int_{0}^{\min(u,4R^{2})} \theta(\sqrt{v})\ud v
\end{equation*}
If $\psi$ designates the time-1-flow associated to $H$, for all $r\in [0,R]$ and any $x\in B_{R}^{2n-2}$
such that $|x|=r$, $\psi(x)$ is the image of $x$ by some 0-centered rotation of angle
$-2h'(r^{2}) = \theta(r)$. Thus $\psi(x)\not\in B_{R}^{2n-2}$ and $\psi(B_{R}^{2n-2})\cap B_{R}^{2n-2}=\emptyset$.
Nevertheless, $\psi$ is not well defined as $H$ is not smooth in the neighborhood of $x=0$ and $|x|=2R$.
For every small $\varepsilon>0$,
we then construct a family of smooth $h_{\varepsilon}:[0,+\infty)\rightarrow\R$ approximating $h$ in the following way
(see Figure~\ref{fig:function}):
there exists $\delta=\delta(\varepsilon)\in (0,2R^2)$ such that
\begin{itemize}
\item $h_{\varepsilon}$ is compactly supported on $[0,4R^2 + \delta]$,
\item $h'(u) \leq h_{\varepsilon}'(u) \leq \frac{\pi}{2}$ for all $u\in [0,+\infty)$,
\item $h_{\varepsilon}(u) = \pi R^{2}+\varepsilon - \frac{\pi}{2}u$ for all $u\in\left[0,\frac{\delta}{2}\right]$,
\item $h_{\varepsilon}(u)=h(u)+\frac{\varepsilon}{2}$ for all $u\in[\delta,4R^2-\delta]$.
\end{itemize}
Hamiltonians $H^{\varepsilon}(x):=-h_{\varepsilon}(|x|^2)$ are smooth functions so
their time-1-flow $\psi^{\varepsilon}$ are well defined.
As $H^{\varepsilon}\leq 0$, $c(1,\psi^{\varepsilon}) = 0$.
The only fixed point with non-zero action is $0$ so
$c(\mu,\psi) \leq -H^{\varepsilon}(0) = -H(0)+\varepsilon$ ($0$ has action $-H^{\varepsilon}(0)$) and
\begin{equation*}
 -H(0) = h(0) =
\frac{1}{2}\int_{0}^{4R^{2}} \theta(\sqrt{v})\ud v+\varepsilon.
\end{equation*}
Changing the variable $x=\frac{\sqrt{v}}{2R}$ and writing $\theta(s) = 2\arccos \left(\frac{s}{2R}\right)$,
\begin{equation*}
\frac{1}{2}\int_{0}^{4R^{2}} \theta(\sqrt{v})\ud v = 8R^{2}\int_{0}^{1}x\arccos(x)\ud x
\end{equation*}
Then, an integration by parts and writing $x=\sin\alpha$ give
\begin{equation*}
\int_{0}^{1}x\arccos(x)\ud x =
\frac{1}{2}\int_{0}^{1}\frac{x^{2}}{\sqrt{1-x^{2}}}\ud x =
\frac{1}{2}\int_{0}^{\frac{\pi}{2}}\sin^{2} \alpha\ud\alpha = \frac{\pi}{8},
\end{equation*}
thus $c(\mu,\psi^{\varepsilon})\leq \pi R^2 + \varepsilon$ as expected.

Now, from the family $\psi^{\varepsilon}$
we deduce a second, $\varphi^{\varepsilon}\in\hamc(T^{*}C\times\R^{2n-2})$, satisfying:
\begin{enumerate}
\item\label{item:cap} $\varphi^{\varepsilon}\left(C\times\left(-\frac{1}{\varepsilon},\frac{1}{\varepsilon}\right)
\times B_{R}^{2n-2}\right)
\cap \left(C\times\left(-\frac{1}{\varepsilon},\frac{1}{\varepsilon}\right)\times B_{R}^{2n-2}\right)
=\emptyset$,
\item\label{item:zero} $c(1,\varphi^{\varepsilon})=0$,
\item\label{item:cmu} $\forall\varepsilon >0$, $c(\mu,\varphi^{\varepsilon})\leq \pi R^2 + \varepsilon$,
\end{enumerate}
where $\mu$ is the orientation class of the compactified space $C\times\mathds{S}^1\times\mathds{S}^{2n-2}$.
Let $U^{\varepsilon}:=C\times\left(-\frac{1}{\varepsilon},\frac{1}{\varepsilon}\right)\times B_{R}^{2n-2}$ and
$\chi:\R\rightarrow [0,1]$ be a smooth compactly supported function with
$\chi|_{\left[-\frac{1}{\varepsilon},\frac{1}{\varepsilon}\right]}\equiv1$.
We then introduce the compactly supported negative Hamiltonians
$K^{\varepsilon}:C\times\R\times\R^{2n-2}\rightarrow\R$ defined by:
\begin{equation*}
K^{\varepsilon}(q_{1},p_{1},x) := \chi(p_{1})H^{\varepsilon}(x),\qquad
\forall (q_{1},p_{1},x)\in C\times\R\times\R^{2n-2},
\end{equation*}
so that $\varphi^\varepsilon(q_{1},p_{1},x) = (q_{1},p_{1},\psi^\varepsilon(x))$ for
$p_{1}\in \left(-\frac{1}{\varepsilon},\frac{1}{\varepsilon}\right)$, thus
$\varphi^\varepsilon(U^\varepsilon)\cap U^\varepsilon = \emptyset$ as wanted.
Moreover, since $K^{\varepsilon}$ is negative, $c(1,\varphi^{\varepsilon})=0$.
The function $\chi$ can be chosen so that it is even and decreasing inside $\supp\chi\cap (\frac{1}{\varepsilon},+\infty )$ with
an arbitrarily small derivative.
For $|p_{1}|>\frac{1}{\varepsilon}$ and $(q_1,p_1,x)\in\supp\varphi^{\varepsilon}$,
the $q_1$-coordinate of $\varphi^\varepsilon(q_{1},p_{1},x)$ is thus slightly different from $q_1$.
Thus, the only fixed points with a non-zero action are the $(q_1,p_1,0)$'s for $|p_{1}|\leq\frac{1}{\varepsilon}$.
The action is still given by $-K^{\varepsilon}(0)=h(0)+\varepsilon=\pi R^{2}+\varepsilon$,
thus $c(\mu,\varphi^{\varepsilon})\leq \pi R^{2}+\varepsilon$.

Take the contact lift of the previous family:
$\widehat{\varphi^{\varepsilon}}\in\contoc(T^{*}C\times\R^{2n-2}\times S^1)$.
Property (\ref{item:cap}) of $\varphi^{\varepsilon}$ implies
\begin{equation}\label{eq:cap}
\widehat{\varphi^{\varepsilon}}\left(C\times\left(-\frac{1}{\varepsilon},\frac{1}{\varepsilon}\right)
\times B_{R}^{2n-2}\times S^1\right)
\cap \left(C\times\left(-\frac{1}{\varepsilon},\frac{1}{\varepsilon}\right)\times B_{R}^{2n-2}\times S^1\right)
=\emptyset.
\end{equation}
On the one hand, Proposition~\ref{prop:ccont} (\ref{item:ccontcomp})
and property (\ref{item:zero}) of $\varphi^{\varepsilon}$ gives
\begin{equation*}
c\left(1,\widehat{\varphi^{\varepsilon}}\right) = c(1,\varphi^{\varepsilon}) = 0.
\end{equation*}
Thus, if $\ud z$ denotes the orientation class of $S^1$, Proposition~\ref{prop:ccont} (\ref{item:ccontdual}) gives
\begin{equation}\label{eq:zero}
c\left(\mu\otimes\ud z,\left(\widehat{\varphi^{\varepsilon}}\right)^{-1}\right) =0.
\end{equation}
On the other hand,
Proposition~\ref{prop:ccont} (\ref{item:ccontcomp}) and property (\ref{item:cmu}) of $\varphi^{\varepsilon}$
gives,
\begin{equation}\label{eq:cmu}
c\left(\mu\otimes\ud z,\widehat{\varphi^{\varepsilon}}\right) =
c(\mu,\varphi^{\varepsilon}) \leq
\pi R^2 + \varepsilon.
\end{equation}
Equations (\ref{eq:zero}) and (\ref{eq:cmu}) then imply
\begin{equation*}
\left\lceil c\Big(\mu\otimes\ud z,\widehat{\varphi^{\varepsilon}}\Big) \right\rceil +
\left\lceil c\left(\mu\otimes\ud z,\left(\widehat{\varphi^{\varepsilon}}\right)^{-1}\right) \right\rceil
\leq \left\lceil\pi R^2 + \varepsilon\right\rceil.
\end{equation*}
Thus, since $\widehat{\varphi^{\varepsilon}}$ verifies (\ref{eq:cap}),
\begin{equation*}
\gamma \left(U^\varepsilon\times S^{1}\right) 
\leq
\left\lceil \pi R^2  + \varepsilon\right\rceil.
\end{equation*}
Since $\pi R^{2}\not\in\Z$, $x\mapsto\lceil x\rceil$ is continuous at $\pi R^{2}$ and
any open bounded set $V\subset T^{*}C\times B_{R}^{2n-2}\times S^{1}$ is included in $U^\varepsilon\times S^{1}$
for a small $\varepsilon$, we conclude that
\begin{equation*}
\gamma \left(T^{*}C\times B_{R}^{2n-2}\times S^{1}\right) \leq
\left\lceil \pi R^2 \right\rceil.
\end{equation*}
\end{proof}

\subsection{Linear symplectic invariance} \label{se:sympinv}

A symplectomorphism $\varphi:T^{*}(\R^{n}\times\T^{k})\rightarrow T^{*}(\R^{n}\times\T^{k})$
will be called \emph{linear} when it can be lifted to a linear map
$\widetilde{\varphi}:\R^{2(n+k)}\rightarrow\R^{2(n+k)}$.
Throughout this subsection, we fix a linear symplectomorphism
$\varphi:T^{*}(\R^{n}\times\T^{k})\rightarrow T^{*}(\R^{n}\times\T^{k})$
of the form
\begin{equation*}
\varphi(q_1,q_2) = (\varphi_1(q_1),\varphi_2(q_2)),\qquad \forall (q_1,q_2)\in\R^{2n+k}\times\T^{k},
\end{equation*}
for some linear maps
$\varphi_{1}:\R^{2n+k}\rightarrow\R^{2n+k}$ and
$\varphi_{2}:\T^{k}\rightarrow\T^{k}$.
Let $B$ be either $\mathds{S}^{2n}\times\T^{k}$
or $\mathds{S}^{2n+k}$,
such that $B\times\T^{k}$ is a compactification of $T^{*}(\R^{n}\times\T^{k})$.
We denote by $\mu\in H^{*}(B)$ the orientation class of $B$.

\begin{prop} \label{prop:sympinv}
For any open subset $U\subset T^{*}(\R^{n}\times\T^{k})$ and for any non-zero
$\alpha\in H^{*}(\T^{k})$,
we have
\begin{equation*}
c(\mu\otimes\alpha, \varphi(U)) = c(\mu\otimes\varphi_{2}^{*}\alpha, U).
\end{equation*}
\end{prop}

Proposition~\ref{prop:sympinv} is a consequence of the following statement:

\begin{lem} \label{prop:hamconj}
For any $\psi\in\hamc(T^{*}(\R^{n}\times\T^{k}))$ and for any non-zero
$\alpha\in H^{*}(\T^{k})$,
\begin{equation*}
c(\mu\otimes\alpha, \psi) = c(\mu\otimes\varphi_{2}^{*}\alpha, \varphi^{-1}\circ\psi\circ\varphi).
\end{equation*}
\end{lem}

In order to prove Lemma \ref{prop:hamconj}, we will need suitable generating functions for $\psi$ and
$\varphi^{-1}\circ\psi\circ\varphi$:

\begin{lem} \label{lem:hamconj}
Let $F_{1}: B\times\T^{k}\times\R^N\rightarrow \R$ 
be a generating function of
$\psi\in\hamc(T^{*}(\R^{n}\times\T^{k}))$.
There exists
a diffeomorphism $\varphi_1':B\rightarrow B$
such that, if 
$\Phi : B\times \T^{k}\times\R^N\rightarrow B\times \T^{k}\times\R^N$
denotes the diffeomorphism $\Phi(q_1,q_2;\xi)=(\varphi_1'(q_1),\varphi_2(q_2);\xi)$,
then $F_{2} := F_{1}\circ\Phi$
is a generating function of $\varphi^{-1}\circ\psi\circ\varphi\in\hamc(T^{*}(\R^{n}\times\T^{k}))$.
\end{lem}

\begin{proof}[Proof of Lemma \ref{lem:hamconj}]
Let $\psi\in\hamc(T^{*}(\R^{n}\times\T^{k}))$ and $R>0$ such that $\supp\psi\subset B_{R}^{2n+k}\times\T^k$.
Since $\varphi_1:\R^{2n+k}\rightarrow\R^{2n+k}$ is linear and invertible,
one can find a diffeomorphism $\varphi_1':\R^{2n+k}\rightarrow\R^{2n+k}$ such that
$\varphi_1'(x)=\varphi_1(x)$ for all $x\in B_{R}^{2n+k}\cup \varphi_1^{-1}(B_R^{2n+k})$ and
$\varphi_1'(x)=(x_1,\ldots ,x_{2n+k-1},\pm x_{2n+k})$ outside some compact set,
thus $\varphi_1'$ can naturally be extended into a diffeomorphism $\varphi_1':B\rightarrow B$.
Let $F_{1}: B\times\T^{k}\times\R^N\rightarrow \R$ be a generating function of
$\psi$,
we then define the diffeomorphism $\Phi : B\times \T^{k}\times\R^N\rightarrow B\times \T^{k}\times\R^N$
by $\Phi(q_1,q_2;\xi):=(\varphi_1'(q_1),\varphi_2(q_2);\xi)$.
The function $F_2:=F_1\circ\Phi$ is a generating function since
$\frac{\partial F_2}{\partial \xi}=\frac{\partial F_1}{\partial \xi}\circ\Phi$.
Let $(q;\xi)\in\Sigma_{F_2}$ and $q_0:=(\varphi_1'\times\varphi_2)(q)$.
First, let us assume that $q\in \varphi^{-1}(B_{R}^{2n+k}\times\T^k)$ (so $q_0 = \varphi (q)$),
since $(q_0;\xi)\in\Sigma_{F_1}$, there exists $x_0\in T^{*}(\R^{n}\times\T^{k})$ such that
\begin{equation*}
(q_0;\partial_{q}F_1(q_0;\xi)\cdot v)=
\left(\frac{x_0+\psi(x_0)}{2};\la J(\psi(x_0)-x_0),v \ra \right),
\qquad\forall v\in \R^{2(n+k)}.
\end{equation*}
Let $x=\varphi^{-1}(x_0)$.
On the one hand,
by linearity of $\varphi^{-1}$,
\begin{equation*}
q=\frac{x+\varphi^{-1}\circ\psi\circ\varphi(x)}{2},
\end{equation*}
on the other hand,
$\partial_{q}F_2(q_0;\xi)\cdot v = \la J(\psi(x_0)-x_0),\varphi(v) \ra$ for all $v\in\R^{2(n+k)}$
and $\varphi^{-1}$ is a linear symplectomorphism, thus
\begin{equation*}
\partial_{q}F_2(q_0;\xi)\cdot v = \la J\varphi^{-1}(\psi(x_0)-x_0),v \ra
= \la J(\varphi^{-1}\circ\psi\circ\varphi(x)-x),v \ra,
\qquad\forall v\in \R^{2(n+k)}.
\end{equation*}
Now, let us assume that $q\not\in \varphi^{-1}(B_{R}^{2n+k}\times\T^k)$.
If $q$ is at infinity,
then $\partial_{q}F_{2}(q;\xi) = 0$ since $\ud F_1 = 0$ at any point at infinity.
If $q\in \R^{2n+k}\times\T^{k}$, let $x_0\in T^{*}(\R^{n}\times\T^{k})$ be associated to $q_0$
as above. Since $\frac{x_0+\psi(x_0)}{2}\not\in B_{R}^{2n+k}\times\T^k$,
necessarily, $\psi(x_0)=x_0\not\in B_{R}^{2n+k}\times\T^k$ so $\partial_{q}F_{1}(q_0;\xi)=0$ and
$(q_0;\xi)$ is a critical value of $F_1$.
Hence $(q;\xi)$ is a critical value of $F_2$ and
\begin{equation*}
(q;\partial_{q}F_2(q;\xi))= (q;0) =
\left(\frac{x+\varphi^{-1}\circ\psi\circ\varphi(x)}{2};J(\psi(x)-x) \right),
\end{equation*}
where $x=(\varphi_1'\times\varphi_2)^{-1}(x_0)\not\in\supp(\varphi^{-1}\circ\psi\circ\varphi)$.

Conversely, if $x\in T^{*}(\R^{n}\times\T^k)$, the associated $(q;\xi)\in\Sigma_{F_2}$
is given by $((\varphi_1'\times\varphi_2)^{-1}(q_0);\xi)$ where $(q_0;\xi)\in\Sigma_{F_1}$ is
associated to $x_0=(\varphi_1'\times\varphi_2)(x)\in T^{*}(\R^{n}\times\T^k)$.
\end{proof}

\begin{proof}[Proof of Lemma~\ref{prop:hamconj}]
Let $F_1:B\times\T^k\times\R^N\rightarrow\R$ be a generating function of
$\psi\in\hamc(T^{*}(\R^{n}\times\T^{k}))$.
Let $\varphi_1':B\rightarrow B$ and $\Phi:B\times\T^k\times\R^N\rightarrow B\times\T^k\times\R^N$
be the diffeomorphisms defined by Lemma~\ref{lem:hamconj} such that $F_2:=F_1\circ\Phi$
is a generating function of $\varphi^{-1}\circ\psi\circ\varphi$.
Let us denote by $E_{1}$ and $E_{2}$ the domains of the generating functions
$F_{1}$ and $F_{2}$ respectively.
For all $\lambda \in \R$, $\Phi$ gives a diffeomorphism
of sublevel sets $\Phi:E_{2}^{\lambda}\rightarrow E_{1}^{\lambda}$.
In particular, it induces an homology isomorphism
$ \Phi_{*} : H_{*}(E_{2}^{\lambda}, E_{2}^{-\infty}) \rightarrow H_{*}(E_{1}^{\lambda}, E_{1}^{-\infty})$.
We thus have the following commutative diagram:
\[
\xymatrix{
H^{l}(B\times\T^{k}) \ar[d]^{(\varphi_1'\times\varphi_2)^{*}}  \ar[r]^{T_{1}} &H^{l+q}(E_{1}^{\infty}, E_{1}^{-\infty}) \ar[d]^{\Phi^{*}} \ar[r]^{i_{1,\lambda}^{*}} &H^{l+q}(E_{1}^{\lambda}, E_{1}^{-\infty}) \ar[d]^{\Phi^{*}} \\
H^{l}(B\times\T^{k}) \ar[r]^{T_{2}} &H^{l+q}(E_{2}^{\infty}, E_{2}^{-\infty}) \ar[r]^{i_{2,\lambda}^{*}} &H^{l+q}(E_{2}^{\lambda}, E_{2}^{-\infty})
}
\]
where the $T_{j}$'s denote the isomorphisms induced by the K\"unneth formula (\ref{eq:kun})
and the $i^{*}_{j,\lambda}$'s are the morphisms induced by the inclusions
$i_{j,\lambda}:(E_{j}^{\lambda},E_{j}^{-\infty})\hookrightarrow (E_{j}^{\infty},E_{j}^{-\infty})$.
The commutativity of the right square is clear.
As for the left square, it commutes because $\pi\circ\Phi = (\varphi_1'\times\varphi_2)\circ\pi$,
where $\pi:B\times\T^{k}\times\R^{N}\rightarrow B\times\T^{k}$ is the canonical projection.
Let $\alpha$ be a non-zero class of $H^{l}(\T^{k})$ and $\mu$ be the orientation class of $H^{*}(B)$.
Since the vertical arrows are isomorphisms,
$i_{1,\lambda}^{*}T_1(\mu\otimes\alpha)$ is non-zero if and only if
$i_{2,\lambda}^{*}T_2(\varphi_1'\times\varphi_2)^{*}(\mu\otimes\alpha)$ is non-zero.
Since $\varphi_1'$ is a diffeomorphism, $(\varphi_1')^{*}\mu = \pm\mu$,
thus $(\varphi_1'\times\varphi_2)^{*}(\mu\otimes\alpha) = \pm\mu\otimes\varphi_2^{*}\alpha$
and $i_{1,\lambda}^{*}T_1(\mu\otimes\alpha)$ is non-zero if and only if
$\pm i_{2,\lambda}^{*}T_2(\mu\otimes\varphi_2^{*}\alpha)$ is non-zero.
Therefore,
\begin{equation*}
c(\mu\otimes\alpha, F_1) = c(\mu\otimes\varphi_{2}^{*}\alpha, F_2).
\end{equation*}

\end{proof}

\subsection{Reduction lemma}
In this subsection, we work on the space $T^{*}(\R^{m}\times\T^l\times\T^k)\times S^1$
and the points in this space will be denoted by $(q,p,z)$,
where $q=(q_1,q_2)\in(\R^{m}\times\T^l)\times\T^k$
and $p=(p_1,p_2)\in\R^{m+l}\times\R^k$.
Let $B$ be a compactification of $T^{*}(\R^{m}\times\T^l)$.
Given any open set $U\subset T^{*}(\R^{m}\times\T^l\times\T^k)\times S^1$ and any point $w\in\T^k$,
the \emph{reduction} $U_{w}\subset T^{*}(\R^{m}\times\T^l)\times S^1$ at $q_2=w$ is defined by
\begin{equation*}
U_{w}:=\pi(U\cap\{q_2 = w\}),
\end{equation*}
where $\pi:T^{*}(\R^{m}\times\T^l)\times\{w\}\times\R^{k}\times S^1\rightarrow T^{*}(\R^{m}\times\T^l)\times S^1$
is the canonical projection.

\begin{lem} \label{lem:cred}
Let $\mu$ be the orientation class of $B\times\mathds{S}^k\times S^1$ and $1$ be the generator of $H^{0}(\T^k)$.
For any open bounded set $U\subset T^{*}(\R^{m}\times\T^l\times\T^k)\times S^1$
and any $w\in\T^{k}$,
\begin{equation*}
c(\mu\otimes 1, U) \leq \gamma (U_{w}).
\end{equation*}
\end{lem}

It is an extension to the contact case of Viterbo-Bustillo's reduction lemma
\cite[Prop. 2.4]{Bus} and \cite[Prop. 5.2]{Vit}.
We will follow Bustillo's proof as close as 
contact structure allows us to do.

Let $\mu$ be the orientation class of $B\times\mathds{S}^k\times S^1$ and $1$ be the generator of $H^{0}(\T^k)$
and fix an open bounded set $U\subset T^{*}(\R^{m}\times\T^l\times\T^k)\times S^1$
and a point $w\in\T^{k}$.
Remark that one can write $\mu=\mu_1\otimes\mu_2$, where $\mu_1$ and $\mu_2$ are the orientation classes
of $B\times S^1$ and $\mathds{S}^k$ respectively.
By definition of the contact invariants, it is enough to show that, given any $\psi\in\contoc(U)$ and any 
$\varphi\in\contoc(T^{*}(\R^{m}\times\T^l)\times S^1)$ such that $\varphi (U_{w})\cap U_{w}=\emptyset$,
\begin{equation*}
\lceil c(\mu\otimes 1, \psi)\rceil \leq
\left\lceil c\big(\mu_1,\varphi\big)\right\rceil + \left\lceil c\left(\mu_1,\varphi^{-1} \right) \right\rceil.
\end{equation*}
Thus, we fix a contact isotopy $\psi_{t}$ defined on $T^{*}(\R^{m}\times\T^l\times\T^k)\times S^1$ and
compactly supported in $U$ such that $\psi_0=\id$ and $\psi_1=:\psi\in\contoc(U)$
and we fix a contactomorphism $\varphi\in\contoc(T^{*}(\R^{m}\times\T^l)\times S^1)$ such that
$\varphi (U_{w})\cap U_{w}=\emptyset$.
Let $F^{t}:(B\times\mathds{S}^k\times\T^k\times S^1)\times\R^{N}\rightarrow\R$
be a continuous family of generating functions for the Legendrians
$L^{t}:=L_{\psi_{t}}\subset J^{1}(B\times\mathds{S}^k\times\T^k\times S^1)$
given by Theorem~\ref{thm:contgfqi}, $F:=F^1$ and
$K:(B\times S^{1})\times \R^{N'}\rightarrow\R$ be a generating function of $\varphi$.
By the uniqueness statement of Theorem~\ref{thm:contgfqi}, one may suppose that
$F^{0}(x;\xi)=Q(\xi)$ where $Q:\R^N\rightarrow\R$ is a non-degenerated quadratic form without loss of generality.
Recall that $F^t_{w}:(B\times\mathds{S}^k\times S^1)\times\R^{N}\rightarrow\R$ denotes the function
$F^t_{w}(q_1,p,z;\xi):=F^t(q_1,w,p,z;\xi)$ and let
$\widetilde{K}:(B\times\mathds{S}^k\times S^1)\times\R^{N'}\rightarrow\R$ be the generating function defined by
$\widetilde{K}(x,y,z;\eta) := K(x,z;\eta)$.
In order to prove Lemma~\ref{lem:cred}, we will use the following

\begin{lem}\label{lem:technical}
Given $t\in [0,1]$, let $c^{t} := c(\mu, F^{t}_{w}-\widetilde{K})$
which is a continuous $\R$-valued function. Then we have the following alternative:
\begin{itemize}
\item either $\forall t\in [0,1]$, $c^{t}\not\in\Z$
\item or $\exists \ell \in \Z$ such that $\forall t\in [0,1]$, $c^{t} = \ell$.
\end{itemize}
In particular,
\begin{equation*}
\left\lceil c\left(\mu,-\widetilde{K}\right)\right\rceil = 
\left\lceil c\left(\mu, F_{w}-\widetilde{K}\right)\right\rceil.
\end{equation*}
\end{lem}

\begin{proof}[Proof of Lemma~\ref{lem:technical}]
The reduced function $F^t_w$ generates
$L^{t}_{w}\subset J^{1}(B\times\mathds{S}^k\times S^1)$. 
$L^{t}$ is the image of the immersion (plus points in the $0$-section at infinity):
\begin{equation*}
\Gamma_{\psi^{t}}(q,p,z) =
\left( \frac{q+Q^{t}}{2},\frac{e^{\theta^{t}}p+P^{t}}{2},z
;P^{t}-e^{\theta^{t}}p,q-Q^{t},e^{\theta^{t}}-1
;\frac{1}{2}\left(e^{\theta^{t}}p+P^{t}\right)\left(q-Q^{t}\right)+Z^{t}-z\right),
\end{equation*}
writing $\psi_{t}(q,p,z) = (Q^{t},P^{t},Z^{t})$.
Therefore, $L^{t}_{w}$ is the set of points (plus points in the $0$-section at infinity):
\begin{equation*}
\left( \frac{q_{1}+Q_{1}^{t}}{2},\frac{e^{\theta^{t}}p+P^{t}}{2},z
;P_{1}^{t}-e^{\theta^{t}}p_{1},q-Q^{t},e^{\theta^{t}}-1
;\frac{1}{2}\left(e^{\theta^{t}}p+P^{t}\right)\left(q-Q^{t}\right)+Z^{t}-z\right)
\end{equation*}
for points $(q,p,z)$ that verify $\frac{q_{2}+Q_{2}^{t}}{2}=w$.
In the remaining paragraphs, we will use notations $\mathfrak{p}=\frac{e^{\theta}p+P}{2}$
and $\mathfrak{q}=\frac{q+Q}{2}$.

Suppose there exists $\ell\in\Z$ and $t_{0}\in[0,1]$ such that $c^{t_{0}}=\ell$.
Then it is enough to prove that $t\mapsto c^{t}$ is locally constant.
In order to do so, we will follow Bustillo's proof \cite[lemma 2.8]{Bus}.
Let $\left(\mathfrak{q}^{t_0}_{1},\mathfrak{p}^{t_0},z^{t_0};\xi^{t_0},\eta^{t_0}\right)$ be the critical point of
\begin{equation*}
(F^{t_0}_{w}-\widetilde{K})(\mathfrak{q}_{1},\mathfrak{p},z;\xi,\eta)=
F^{t_0}(\mathfrak{q}_{1},w,\mathfrak{p},z;\xi) - K(\mathfrak{q}_{1},\mathfrak{p}_{1},z;\eta)
\end{equation*}
associated to the min-max value $c^{t_0}=\ell$.
By continuity of the min-max critical point, we may suppose that $K$ is a Morse function in some
neighborhood of
$\left(\mathfrak{q}^{t_0}_{1},\mathfrak{p}^{t_0}_{1},z^{t_0};\eta^{t_0}\right)$
by perturbing $K$ without changing its value at this point.
Writing $x^{t_0}:=\left(\mathfrak{q}^{t_0}_{1},\mathfrak{p}^{t_0},z^{t_0}\right)$, such a critical point verifies
\begin{equation*}
\frac{\partial F^{t_0}_{w}}{\partial x} = \frac{\partial\widetilde{K}}{\partial x} \text{ and }
\frac{\partial F^{t_0}_{w}}{\partial \xi} = \frac{\partial\widetilde{K}}{\partial \eta} = 0.
\end{equation*}
These equations define two points 
$(x^{t_0},\partial_{x}F^{t_0}_{w},F^{t_0}_{w})=
(x^{t_0},\partial_{x}\widetilde{K},F^{t_0}_{w})\in L^{t_0}_{w}$
and $(x^{t_0},\partial_{x}\widetilde{K},\widetilde{K})\in L_{\varphi}\times 0_{T^{*}\mathds{S}^{k}}$
which only differ in
the last coordinate by a $\ell\frac{\partial}{\partial a}$ factor:
\begin{equation*}
\left(x^{t_0},\partial_{x}F^{t_0}_{w},F^{t_0}_{w}\right) = \left(x^{t_0},\partial_{x}\widetilde{K},\widetilde{K}+\ell\right) 
= \left(x^{t_0},\partial_{x}\widetilde{K},\widetilde{K}\right) + \ell\frac{\partial}{\partial a}.
\end{equation*}
We will denote by $(q^{t_0},p^{t_0},z^{t_0})\in T^{*}(\R^{m}\times\T^{k+l})\times S^1$ the point whose image is
$\Gamma_{\psi_{t_0}}(q^{t_0},p^{t_0},z^{t_0}) =
((\mathfrak{q}_{1}^{t_0},w,\mathfrak{p}^{t_0},z^{t_0};\xi^{t_0}),\partial_{x}F^{t_0},F^{t_0})$
and we will denote $(Q^{t_0},P^{t_0},Z^{t_0})=\psi_{t_0}(q^{t_0},p^{t_0},z^{t_0})$.
Since $(x^{t_0},\partial_{x}\widetilde{K},\widetilde{K})\in L_{\varphi}\times 0_{T^{*}\mathds{S}^{k}}$,
$\partial_{p_2}\widetilde{K} = 0$ so $q^{t_0}_{2} = Q^{t_0}_{2}$ ($=w$).

Remark that $\varphi(U_{w})\cap U_{w} = \emptyset$
together with $\left(x^{t_0},\partial_{x}F^{t_0}_{w},F^{t_0}_{w}\right)\in
\left(L_{\varphi}+\ell\frac{\partial}{\partial a}\right)\times 0_{T^{*}\mathds{S}^{k}}$
implies
that either $\left(q^{t_0}_{1},p^{t_0}_{1},z^{t_0}\right)\not\in U_{w}$
or $\left(Q^{t_0}_{1},P^{t_0}_{1},Z^{t_0}\right)\not\in U_{w}$.
In order to see it, we go back to the definition of generating function on $T^{*}(\R^{m}\times\T^{l})\times S^1$
given in Subsection~\ref{se:gfhc}.
Let $\pi:J^{1}\R^{m+l}\rightarrow T^{*}(\R^{m}\times\T^{l})\times S^1$
be the quotient projection and
consider the $\Z^{l}\times\Z$-equivariant lift of $\varphi$:
$\widetilde{\varphi}\in\Cont(J^{1}\R^{m+l})$ and
$\widetilde{U_{w}}:=\pi^{-1}(U_{w})\subset J^{1}\R^{m+l}$.
Since $\varphi(U_{w})\cap U_{w} = \emptyset$, we have that
$\widetilde{\varphi}(\widetilde{U_{w}})\cap \widetilde{U_{w}} = \emptyset$
so $L_{\widetilde{\varphi}}\cap \widehat{\tau}(\widetilde{U_{w}}\times \widetilde{U_{w}}\times \R) = \emptyset$.
But $\widetilde{U_{w}}+\frac{\partial}{\partial z}=\widetilde{U_{w}}$
and $\widehat{\tau}(x,X+\frac{\partial}{\partial z},\theta) = \widehat{\tau}(x,X,\theta)+\frac{\partial}{\partial a}$
for all $(x,X,a)\in J^{1}\R^{m+l}\times J^{1}\R^{m+l}\times \R$,
intersection $L_{\widetilde{\varphi}}\cap \widehat{\tau}(\widetilde{U_{w}}\times \widetilde{U_{w}}\times \R) = \emptyset$
is thus equivalent to
\begin{equation*}
\left(L_{\widetilde{\varphi}}+\ell\frac{\partial}{\partial a}\right)\cap
\widehat{\tau}\left(\widetilde{U_{w}}\times \widetilde{U_{w}}\times \R\right) = \emptyset
\end{equation*}
(definition of $\widehat{\tau}:J^{1}\R^{m+l}\times J^{1}\R^{m+l}\times \R \rightarrow  J^{1}\R^{2(m+l)+1}$
is given in Subsection~\ref{se:gfhc}).
Hence, given any point $(u,v,a)\in L_{\widetilde{\varphi}}+\ell\frac{\partial}{\partial a}$,
the corresponding $(x,X,\theta) = \widehat{\tau}^{-1}(u,v,a)$ verifies that either $x\not\in \widetilde{U_{w}}$
or $X\not\in \widetilde{U_{w}}$. This property descends to quotient:
$\left(x^{t_0},\partial_{x}F^{t_0}_{w},F^{t_0}_{w}\right)\in
\left(L_{\varphi}+\ell\frac{\partial}{\partial a}\right)\times 0_{T^{*}\mathds{S}^{k}}$
implies
that either $(q^{t_0}_{1},p^{t_0}_{1},z^{t_0})\not\in U_{w}$
or $(Q^{t_0}_{1},P^{t_0}_{1},Z^{t_0})\not\in U_{w}$.

Since $q^{t_0}_{2}=Q^{t_0}_{2}=w$, it follows that either $(q^{t_0},p^{t_0},z^{t_0})\not\in U$ 
or $(Q^{t_0},P^{t_0},Z^{t_0})\not\in U$. Since $\psi_{t_0}$ has its support in $U$, they both imply that
\begin{equation*}
\left(Q^{t_0},P^{t_0},Z^{t_0}\right)=\psi_{t_0}\left(q^{t_0},p^{t_0},z^{t_0}\right) 
= \left(q^{t_0},p^{t_0},z^{t_0}\right)\not\in U.
\end{equation*}
Hence, $\left(q^{t_0},p^{t_0},z^{t_0}\right)$ is outside the support of $\psi_{t_0}$,
thus, the associated point $(\mathfrak{q}^{t_0}_{1},w,\mathfrak{p}^{t_0},z^{t_0};\xi^{t_0})\in\Sigma_{F^{t_0}}$
is critical of value $F^{t_0}(\mathfrak{q}^{t_0}_{1},w,\mathfrak{p}^{t_0},z^{t_0};\xi^{t_0})=0$.
Thus, we have seen that $m^{t_0}:=\left(\mathfrak{q}^{t_0}_{1},\mathfrak{p}^{t_0},z^{t_0};\xi^{t_0},\eta^{t_0}\right)$
verifies
\begin{equation*}
\frac{\partial F^{t_0}_{w}}{\partial x} = \frac{\partial\widetilde{K}}{\partial x} = 0,\;
\frac{\partial F^{t_0}_{w}}{\partial \xi} = \frac{\partial\widetilde{K}}{\partial \eta} = 0 \text{ and }
F^{t_0}_{w} = 0,
\end{equation*}
so it is a critical point of $-\widetilde{K}$, as wished, with the same critical value
$-\widetilde{K} = F^{t_0}_{w}-\widetilde{K}$.

Let $t\mapsto m^{t}$ be the continuous path of critical value of $t\mapsto F^{t}_{w}-\widetilde{K}$ obtained by min-max.
It remains to show that $c^t=(F^{t}_{w}-\widetilde{K})(m^{t})$ is equal to $\ell$ in some neighborhood of $t_{0}$.
Since $\widetilde{K}$ does not depend on $\xi$,
\begin{equation*}
\frac{\partial F^{t}}{\partial \xi}\left(\mathfrak{q}^{t}_{1},w,\mathfrak{p}^{t},z^{t};\xi^{t}\right) = 0
\end{equation*}
so the point $\left(\mathfrak{q}^{t}_{1},w,\mathfrak{p}^{t},z^{t};\xi^{t}\right)$ remains inside the level set
$\Sigma_{F^{t}}$.
If $H:[0,1]\times (T^{*}(\R^{m}\times\T^l\times\T^k)\times S^1)\rightarrow \R$
denotes the compactly supported Hamiltonian map
associated to $(\psi_{t})$,
\begin{equation*}
\iota_{F^{t_0}}\left(\mathfrak{q}^{t_0}_{1},w,\mathfrak{p}^{t_0},z^{t_0};\xi^{t_0}\right)\in U^{c}\subset (\supp H)^{c},
\end{equation*}
and $(\supp H)^{c}$ is an open set so, for all $t$ in a small neighborhood of $t_{0}$,
$\iota_{F^t}\left(\mathfrak{q}^{t}_{1},w,\mathfrak{p}^{t},z^{t};\xi^{t}\right)\in (\supp H)^{c}$
thus $F^{t}\left(\mathfrak{q}^{t}_{1},w,\mathfrak{p}^{t},z^{t};\xi^{t}\right) = 0$ and
$\left(\mathfrak{q}^{t}_{1},\mathfrak{p}^{t},z^{t};\xi^{t}\right)$ remains a critical point of $F^{t}_{w}$.
Thus, in a small neighborhood of $t_0$, since $m^{t}$ is a critical point of $F^{t}_{w}-\widetilde{K}$
and $F^{t}_{w}$ (with a slight abuse of notation),
$t\mapsto\left(\mathfrak{q}^{t}_{1},\mathfrak{p}^{t}_{1},z^{t};\eta^{t}\right)$
is a continuous path of critical value for $K$. But $K$ is a Morse function in some neighborhood of
$\left(\mathfrak{q}^{t_0}_{1},\mathfrak{p}^{t_0}_{1},z^{t_0};\eta^{t_0}\right)$, thus
this continuous path is constant and $K\left(\mathfrak{q}^{t}_{1},\mathfrak{p}^{t}_{1},z^{t};\eta^{t}\right)\equiv-\ell$.

Finally, we have seen that, in some neighborhood of $t_0$,
$F^{t}\left(\mathfrak{q}^{t}_{1},w,\mathfrak{p}^{t},z^{t};\xi^{t}\right)\equiv 0$ and
$K\left(\mathfrak{q}^{t}_{1},\mathfrak{p}^{t}_{1},z^{t};\eta^{t}\right)\equiv-\ell$, thus
\begin{equation*}
c^{t}=F^{t}\left(\mathfrak{q}^{t}_{1},w,\mathfrak{p}^{t},z^{t};\xi^{t}\right) -
K\left(\mathfrak{q}^{t}_{1},\mathfrak{p}^{t}_{1},z^{t};\eta^{t}\right)\equiv\ell.
\end{equation*}

In particular, since $t\mapsto \left\lceil c^{t}\right\rceil$ is constant, one has
\begin{equation*}
\left\lceil c\left(\mu,F^{0}_{w}-\widetilde{K}\right)\right\rceil = 
\left\lceil c\left(\mu, F^{1}_{w}-\widetilde{K}\right)\right\rceil.
\end{equation*}
But $F^{0}_{w}(x;\xi)=Q(\xi)$ where $Q$ is a non-degenerated quadratic form,
so $F^{0}_{w}-\widetilde{K}$ is a stabilization of the generating function $-\widetilde{K}$,
thus Proposition~\ref{prop:minmax} (\ref{item:inv}) implies
$c(\mu,F^{0}_{w}-\widetilde{K}) = c(\mu,-\widetilde{K})$.
\end{proof}

\begin{proof}[Proof of Lemma~\ref{lem:cred}]
By Proposition~\ref{prop:redgf} (\ref{item:ineqred}),
\begin{equation*}
c(\mu\otimes 1, \psi) := c(\mu\otimes 1, F)
\leq c(\mu, F_{w}).
\end{equation*}
The triangular inequality of Proposition~\ref{prop:minmax} (\ref{item:additivity}) applied to
$\mu=\mu\smile 1$ gives us
\begin{equation*}
c\Big(\mu, F_{w}\Big) \leq 
c\left(\mu, F_{w} - \widetilde{K}\right) - c\left(1,-\widetilde{K}\right).
\end{equation*}
By Proposition~\ref{prop:minmax} (\ref{item:duality})
and Proposition~\ref{prop:redgf} (\ref{item:eqred}), we have
$- c(1,-\widetilde{K}) = c(\mu, \widetilde{K}) = c(\mu_1, K)$. Hence
\begin{equation*}
c\Big(\mu\otimes 1, \psi\Big) \leq 
c\left(\mu, F_{w} - \widetilde{K}\right) + c\Big(\mu_1, K\Big),
\end{equation*}
and thus,
\begin{equation*}
\left\lceil c\Big(\mu\otimes 1, \psi\Big)\right\rceil \leq 
\left\lceil c\Big(\mu, F_{w} - \widetilde{K}\Big)\right\rceil 
+\left\lceil c\Big(\mu_1, K\Big)\right\rceil.
\end{equation*}
According Lemma~\ref{lem:technical},
$\lceil c(\mu, F_{w} - \widetilde{K})\rceil = \lceil c(\mu, - \widetilde{K})\rceil$ so
\begin{equation*}
\left\lceil c\Big(\mu\otimes 1, \psi\Big)\right\rceil \leq 
\left\lceil c\Big(\mu, -\widetilde{K}\Big)\right\rceil 
+\left\lceil c\Big(\mu_1, K\Big)\right\rceil.
\end{equation*}
Since $K$ generates $\varphi$,
$\left\lceil c(\mu_1, -K)\right\rceil = 
\lceil c(\mu_1,\varphi^{-1})\rceil$,
according to Proposition~\ref{prop:ccont} (\ref{item:ccontminus}).
Thus Proposition~\ref{prop:redgf} (\ref{item:eqred}) gives
$\lceil c(\mu, -\widetilde{K})\rceil =
\lceil c(\mu_1,\varphi^{-1})\rceil$.
Finally, by definition, $c(\mu_1, K)=c(\mu_1,\varphi)$.
\end{proof}

\section{Contact camel theorem}\label{se:thm}

In this section, we will prove Theorem \ref{thm:ccamel}.
We work on the space $\R^{2n}\times S^1$ in dimension $2n+1>3$,
we denote by $q_1,p_1,\ldots,q_n,p_n,z$ coordinates on $\R^{2n}\times S^1$
so that the Liouville form is given by $\lambda = p\ud q := p_1\ud q_1 +\cdots +p_n\ud q_n$
and the standard contact form of $\R^{2n}\times S^1$ is $\alpha = p\ud q - \ud z$.
Let $\tau_{t}(x) = x+t\frac{\partial}{\partial q_{n}}$ be the contact Hamiltonian flow
of $\R^{2n}\times S^1$ associated to the contact Hamiltonian
$(t,x)\mapsto p_{n}$.

\begin{lem}\label{lem:construction}
Let $R$ and $r$ be two positive numbers and $B_R^{2n}\times S^{1}\subset P_-\times S^{1}$.
If there exists a contact isotopy $(\phi_t)$ of $(\R^{2n}\times S^1,\alpha)$
supported in $[-c/8,c/8]^{2n}\times S^1$ for some $c>0$
such that $\phi_0=\id$, $\phi_1(B_R^{2n}\times S^{1})\subset P_+\times S^{1}$
and $\phi_t(B_R^{2n}\times S^{1})\subset (\R^{2n}\setminus P_r)\times S^{1}$ for all $t\in[0,1]$,
then there exists a smooth family of contact isotopy $s\mapsto (\psi_{t}^{s})$ with
$\psi_{t}^{s}\in\Cont(\R^{2n}\times S^1)$ and $\psi^{s}_{0}=\id$
associated to a smooth family of contact Hamiltonians $s\mapsto (H_{t}^{s})$
supported in $[-c/8,c/8]^{2n-2}\times\R^2\times S^1$
, such that,
for all $s\in [0,1]$, all $t\in\R$ and all $x\in\R^{2n}\times S^1$,
\begin{equation}\label{eq:cpsitran}
\psi^{s}_{c}(x) = x + c\frac{\partial}{\partial q_{n}},
\end{equation}
\begin{equation}\label{eq:cpsiqper}
\psi^{s}_{t}\left(x + c\frac{\partial}{\partial q_{n}}\right) =
\psi^{s}_{t}(x) + c\frac{\partial}{\partial q_{n}},
\qquad\forall t\in\R
\end{equation}
\begin{equation}\label{eq:cpsitper}
\psi^{s}_{t+c} = \psi^{s}_{c}\circ\psi^{s}_{t}.
\end{equation}
Moreover, for all $t\in\R$, $\psi^{0}_{t} = \tau_{t}$ whereas $\psi_{t} := \psi^{1}_{t}$ satisfies
\begin{equation}\label{eq:cpsiBr}
\psi_{t}\left(B^{2n}_{R}\times S^{1}\right) \subset
\left(\R^{2n}\setminus \bigcup_{k\in\Z} \left(P_{r} + kc\frac{\partial}{\partial q_{n}} \right)\right)
\times S^{1},
\qquad\forall t\in \R.
\end{equation}
\end{lem}

\begin{proof}
Assume there exists such a $(\phi_{t})$.
Let $K_t:\R\times(\R^{2n}\times S^1)\rightarrow\R$ be the compactly supported contact Hamiltonian associated
to $(\phi_{t})$.
By hypothesis, $K_{t}$ is supported in $[-c/8,c/8]^{2n}$, thus one can define
its $c\frac{\partial}{\partial q_{n}}$-periodic extension $K_{t}':\R\times(\R^{2n}\times S^1)\rightarrow\R$
and the associated contact isotopy $(\phi_{t}')$.
The contactomorphism $\phi_{t}':\R^{2n}\times S^1\rightarrow\R^{2n}\times S^1$ satisfies
$\phi_{t}'\left(x + c\frac{\partial}{\partial q_{n}}\right) = \phi_{t}'(x) + c\frac{\partial}{\partial q_{n}}$.

\begin{figure}
\begin{small}
\def\svgwidth{\textwidth}
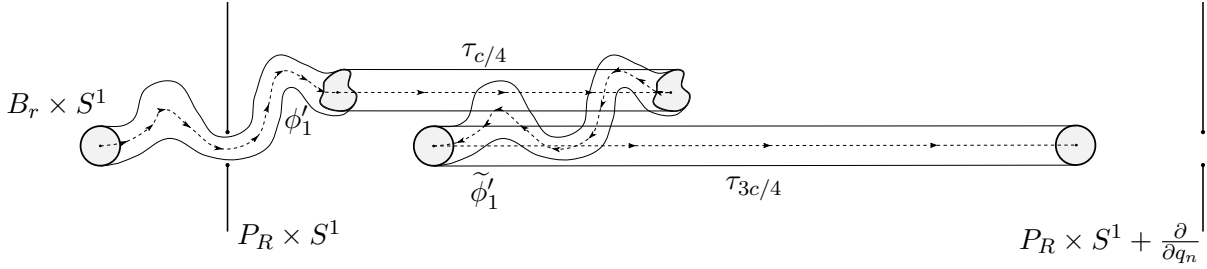
\end{small}
\caption{Construction of $\psi^{1}$.}
\label{fig:contruction}
\end{figure}

For all $s\in[0,1]$, consider the contact isotopy $(\psi^{s}_{t})$ with
$\psi^{s}_{t}\in\Cont(\R^{2n}\times S^1)$ and $\psi^{s}_{0}=\id$
defined as follow
(look also at Figure~\ref{fig:contruction}).
Given $x\in \R^{2n}$, trajectory $\gamma(t)=\psi^{s}_{t}(x)$ first follows $t\mapsto \phi_{t}'(x)$
from $t=0$ to $t=s$. Then $\gamma$ follows $t\mapsto\tau_{t}(\phi_{s}'(x))$ from $t=0$ to $t=1/4$.
Then $t\mapsto\widetilde{\phi}_{t}'(\tau_{1/4}\circ\phi_{s}'(x))$ from $t=0$ to $t=s$,
where $(\widetilde{\phi}_{t})=(\tau_{c/4}\circ\phi_{t}^{-1}\circ\tau_{c/4}^{-1})$
is the contact Hamiltonian flow associated to the translated contact Hamiltonian application
$\widetilde{K}_{t}' = -K_{s-t}'\circ\tau_{-c/4}$. Finally, $\gamma$ follows
$t\mapsto \tau_{t}(\widetilde{\phi}_{s}'\circ\tau_{1/4}\circ\phi_{s}'(x))$ from $t=0$ to $t=3/4$.
We normalize time such that $s\mapsto\psi_{t}^{s}$ gives an isotopy
of smooth contact Hamiltonian flows of
$c\frac{\partial}{\partial t}$-periodic contact Hamitonians $H^{s}_{t}$.

Identity (\ref{eq:cpsitran}) comes from the fact that
$\psi^{s}_{c}=\tau_{3c/4}\circ\widetilde{\phi}_{s}\circ\tau_{c/4}\circ\phi_{s}$
and, by definition of $(\widetilde{\phi}_{t})$,
$\widetilde{\phi}_{s}=\tau_{c/4}\circ\phi_{s}^{-1}\circ\tau_{c/4}^{-1}$.
Identity (\ref{eq:cpsiqper}) comes from the fact that contactomorphism $\psi^{s}_{t}$ is
a composition of $c\frac{\partial}{\partial q_n}$-equivariant contactomorphisms.
Identity (\ref{eq:cpsitper}) is implied by $c\frac{\partial}{\partial t}$-periodicity
of Hamiltonian $H^{s}_{t}$.
Inclusion (\ref{eq:cpsiBr}) comes from the hypothesis on the contact isotopy $(\phi_{t})$.
\end{proof}

Let $r,R>0$ be such that there exists a positive integer $\ell$ satisfying $\pi r^2 < \ell < \pi R^2$
and let $B_R^{2n}\times S^{1}\subset P_-\times S^{1}$.
Suppose by contradiction that
there exists a contact isotopy $(\phi_t)$ of $(\R^{2n}\times S^1,\alpha)$
supported in $[-c/8,c/8]^{2n}\times S^1$ for some $c>0$
such that $\phi_0=\id$, $\phi_1(B_R^{2n}\times S^{1})\subset P_+\times S^{1}$
and $\phi_t(B_R^{2n}\times S^{1})\subset (\R^{2n}\setminus P_r)\times S^{1}$ for all $t\in[0,1]$.
In order to prove Theorem~\ref{thm:ccamel}, it is enough to consider $r\in (\ell-1,\ell)$.
Consider the family of contact isotopy $s\mapsto (\psi_{t}^{s})$ given by Lemma~\ref{lem:construction}
and denote by $(H_{t}^{s})$ the associated family of Hamiltonian 
supported in $[-c/8,c/8]^{2n}\times\R^{2}\times S^1$.
We define $\lambda_{t}^{s}:\R^{2n}\times S^1\rightarrow\R$ by 
$(\psi_{t}^{s})^{*}\alpha=\lambda_{t}^{s}\alpha$.
Let us consider:
\begin{equation*}
\overline{\Psi^{s}}:
\R\times\R\times\R^{2n+1} \rightarrow \R\times\R\times\R^{2n+1},
\end{equation*}
\begin{equation*}
\overline{\Psi^{s}}(t,h,x) = (t,\lambda_{t}^{s}(x)h+H^{s}_{t}\circ\psi^{s}_{t}(x),\psi^{s}_{t}(x)).
\end{equation*}
According to (\ref{eq:cpsitran}), (\ref{eq:cpsiqper}) and (\ref{eq:cpsitper}),
for all $(t,h,x)\in\R\times\R\times\R^{2n+1}$,
\begin{equation*}
\forall k,l\in\Z,\quad
\overline{\Psi^{s}}\left(t+lc,h,x + kc\frac{\partial}{\partial q_{n}}\right) =
\overline{\Psi^{s}}(t,h,x) + lc\frac{\partial}{\partial t} + (k+l)c\frac{\partial}{\partial q_{n}}.
\end{equation*}

Thus $\overline{\Psi^{s}}$ descends to a map
$\Psi^{s}:T^{*}C\times T^{*}(\R^{n-1}\times C)\times S^{1}\rightarrow
T^{*}C\times T^{*}(\R^{n-1}\times C)\times S^{1}$ where $C:=\R/c\Z$.

\begin{lem}
The family $s\mapsto\Psi^{s}$ is a contact isotopy of the contact manifold
$(T^{*}C\times T^{*}(\R^{n-1}\times C)\times S^{1},\ker (\ud z - p\ud q + h\ud t))$.
\end{lem}

\begin{proof}
We write
\begin{equation*}
\psi^{s}_t(q,p,z) = (Q_{t}(q,p,z),P_{t}(q,p,z),Z_{t}(q,p,z)),
\end{equation*}
since
$\ud Z_{t} - P_{t}\ud Q_{t} = (\psi_{t}^{s})^{*}(\ud z - p\ud q) = \lambda_{t}^{s}(\ud z - p\ud q)$,
we have
\begin{equation*}
\left(\Psi^{s}\right)^{*}(\ud z - p\ud q + h\ud t) = 
\ud Z_{t} - P_{t}\ud Q_{t} + \dot{Z}_{t}\ud t - P_{t}\dot{Q}_{t}\ud t + (\lambda_{t}^{s}h+H^{s}_{t}\circ\psi^{s}_{t})\ud t
\end{equation*}
\begin{equation*}
= \lambda_{t}^{s}(\ud z - p\ud q + h\ud t) + \left(\dot{Z}_{t} - P_{t}\dot{Q}_{t}+H^{s}_{t}\circ\psi^{s}_{t}\right)\ud t.
\end{equation*}
But,
since $H^{s}_{t}$ is the contact Hamiltonian of the isotopy $\psi_{t}^{s}$
supported in $[-c/8,c/8]^{2n}\times\R^{2}\times S^1$,
$P_{t}\dot{Q}_{t} - \dot{Z}_{t} = H^{s}_{t}\circ\psi_{t}^{s}$. Finally,
\begin{equation*}
\left(\Psi^{s}\right)^{*}(\ud z - p\ud q + h\ud t) = \lambda_{t}^{s}(\ud z - p\ud q + h\ud t).
\end{equation*}

\end{proof}

For technical reasons, we replace $T^{*}\R^n\times S^1$ by its quotient
$T^{*}(\R^{n-1}\times C)\times S^1$ and consider our $B^{2n}_R\times S^1$
inside this quotient (since $c$ can be taken large while $R$ is fixed, this identification is well defined).
Let us consider the linear symplectic map:
\begin{equation*}
L : T^{*}C\times T^{*}(\R^{n-1}\times C)\rightarrow T^{*}C\times T^{*}(\R^{n-1}\times C),
\end{equation*}
\begin{equation*}
L(t,h,x,q_n,p_{n}) = (q_n-t,-h,x,q_n,p_{n}-h)
\end{equation*}
and denote by $\widehat{L}=L\times \text{id}$ the associated contactomorphism of
$T^{*}C\times T^{*}(\R^{n-1}\times C)\times S^1$.
Let us consider 
\begin{equation*}
\mathcal{U}:=\widehat{L}\left(\Psi^{1}\left(T^{*}C\times B^{2n}_{R}\times S^{1}\right)\right).
\end{equation*}
We compactify the space $T^{*}C\times T^{*}(\R^{n-1}\times C)$
as
\begin{equation*}
C\times\mathds{S}^1\times\mathds{S}^{2n-2}\times C\times\mathds{S}^1
\simeq \mathds{S}^{2n-2}\times\T^{2}\times C^{2}.
\end{equation*}
Let $\mu$ and $\ud z$ be the orientation class of $\mathds{S}^{2n-2}\times\T^{2}$ and $S^1$ respectively and
$\ud q_n$ and $\ud t$ be the canonical base of $H^{1}(C^2)$.

\begin{lem}\label{lem:capineq}
One has the following capacity inequality:
\begin{equation*}
c\left(\mu\otimes\ud t\otimes\ud z, \mathcal{U}\right)
\geq \left\lceil\pi R^2 \right\rceil.
\end{equation*}
\end{lem}

\begin{proof}
Let $\alpha := \mu\otimes\ud t$.
Since $s\mapsto \widehat{L}\circ\Psi^{s}$ is a contact isotopy,
Proposition~\ref{prop:ccU} (\ref{item:ccUinv}) implies
\begin{equation*}
c\left(\alpha\otimes\ud z, \mathcal{U}\right) =
c\left(\alpha\otimes\ud z,
\widehat{L}\left(\Psi^{0}\left(T^{*}C\times B^{2n}_{R}\times S^{1}\right)\right)\right).
\end{equation*}

But $\widehat{L}\left(\Psi^{0}\left(T^{*}C\times B^{2n}_{R}\times S^{1}\right)\right)=
L(\Phi^{0}\left(T^{*}C\times B^{2n}_{R}\right))\times S^{1}$
where $\Phi^{0}:T^{*}C\times T^{*}(\R^{n-1}\times C)\rightarrow T^{*}C\times T^{*}(\R^{n-1}\times C)$
is the linear symplectic map:
\begin{equation*}
\Phi^{0}(t,h,x,q_n,p_{n}) =
(t, h + p_{n},x,q_n+t,p_{n}).
\end{equation*}
Thus, using Proposition~\ref{prop:ccU} (\ref{item:ccUcomp}),
\begin{equation*}
c\left(\alpha\otimes\ud z, \mathcal{U}\right) =
\left\lceil c\left(\alpha,
L\left(\Phi^{0}\left(T^{*}C\times B^{2n}_{R}\right)\right) \right)\right\rceil.
\end{equation*}
In order to conclude, let us show that
\begin{equation*}
c\left(\mu\otimes\ud t,
L\left(\Phi^{0}\left(T^{*}C\times B^{2n}_{R}\right)\right) \right) \geq \pi R^2.
\end{equation*}
By the linear symplectic invariance stated in Proposition~\ref{prop:sympinv},
\begin{equation*}
c\left(\mu\otimes\ud t,
L\circ\Phi^{0}\left(T^{*}C\times B^{2n}_{R}\right) \right) =
c\left(\mu\otimes A^{*}\ud t,
T^{*}C\times B^{2n}_{R}\right),
\end{equation*}
where $A:C^2\rightarrow C^2$ is the linear map $A(t,q_n) = (q_n,q_n+t)$.
We have $A^{*}\ud t = \ud q_n$, therefore
\begin{equation*}
c\left(\mu\otimes\ud t,
L\circ\Phi^{0}\left(T^{*}C\times B^{2n}_{R}\right) \right) =
c\left(\mu\otimes \ud q_n,
T^{*}C\times B^{2n}_{R}\right).
\end{equation*}
The cohomology class $\mu\otimes\ud q_n$ can be seen as the tensor product
of the orientation class $\mu_{1}$ of
the compactification $\mathds{S}^{2n-2}\times\mathds{S}^1\times C$ of $T^{*}(\R^{n-1}\times C)$
by the orientation class $\ud h$ of the compactification $\mathds{S}^1$ of the $h$-coordinate
and the generator $1$ of $H^{0}(C)$ (for the $t$-coordinate).
Indeed, $\ud q_n \in H^{1}(C^2)$ can be identify to
$\ud q_n\otimes 1\in H^{1}(C)\otimes H^{0}(C)$
(writing also $\ud q_n$ for the orientation class of $C$ by a slight abuse of notation).
Hence,
if $\mu_{\mathds{S}^{2n-2}}$ and $\ud p_n$ are the orientation classes of $\mathds{S}^{2n-2}$
and of the compactification $\mathds{S}^1$ of the $p_n$-coordinate respectively, then
\begin{equation*}
\mu\otimes\ud q_n = (\mu_{\mathds{S}^{2n-2}}\otimes\ud h\otimes\ud p_n)\otimes(\ud q_n\otimes 1)
= (\mu_{\mathds{S}^{2n-2}}\otimes\ud q_n\otimes\ud p_n)\otimes\ud h\otimes 1
= \mu_{1}\otimes\ud h\otimes 1.
\end{equation*}
Now, according to Proposition~\ref{prop:cU} (\ref{item:cUred}),
\begin{equation*}
c\left(1\otimes\ud h\otimes\mu_{1},C\times\R\times B^{2n}_{R}\right)
\geq c\left(\mu_{1},B^{2n}_{R}\right).
\end{equation*}
Finally, Proposition~\ref{prop:cU} (\ref{item:cUB}) implies
\begin{equation*}
c\left(\mu\otimes\ud t,
L\left(\Phi^{0}\left(T^{*}C\times B^{2n}_{R}\right)\right) \right)
\geq c\left(\mu_{1},B^{2n}_{R}\right) = \pi R^2.
\end{equation*}
\end{proof}

\begin{proof}[Proof of Theorem~\ref{thm:ccamel}]
Let us apply Lemma~\ref{lem:cred} with
$m=n-1$, $l=k=1$ and the orientation class $\mu\otimes\ud t\otimes\ud z$
to the exhaustive sequence of open bounded subsets
defined by 
$U^{(k)} := \widehat{L}\left(\Psi^{1}\left(C\times(-k,k)\times B^{2n}_{R}\right)\right)$;
taking the supremum among $k>0$, we find:
\begin{equation}\label{eq:applem}
c\left(\mu\otimes\ud t\otimes\ud z, \mathcal{U}\right)
\leq 
\gamma\left(\mathcal{U}_{0}\right),
\end{equation}
where $\mathcal{U}_0\subset T^{*}C\times T^{*}\R^{n-1}\times S^1$ is the reduction of $\mathcal{U}$ at $q_n=0$.
Now $\Psi^{1}\left(T^{*}C\times B^{2n}_{R}\times S^{1}\right)\subset
 T^{*}C\times\bigcup_{t\in[0,c]}\psi_{t}(B^{2n}_{R}\times S^{1})$ 
so 
\begin{equation}\label{eq:subset}
\mathcal{U}\subset \widehat{L}\left(T^{*}C\times\bigcup_{t\in[0,c]}\psi_{t}(B^{2n}_{R}\times S^{1})\right).
\end{equation}
Let $V:= \bigcup_{t\in[0,c]}\psi_{t}\left(B^{2n}_{R}\times S^{1}\right)\cap \{q_{n} = 0\}$
and $\pi :T^{*}\R^{n-1}\times \{0\}\times\R\times S^1\rightarrow T^{*}\R^{n-1}\times S^1$
be the canonical projection.
Since $\widehat{L}$ does not change the $q_n$-coordinate and coordinates of $T^{*}\R^{n-1}$,
inclusion (\ref{eq:subset}) implies
\begin{equation*}
\mathcal{U}_{0} \subset T^{*}C\times \pi(V).
\end{equation*}
But (\ref{eq:cpsiBr}) implies $V \subset \left(B^{2n}_{r}(0)\cap \{ q_n = 0 \}\right)\times S^{1}$ and
$\pi (B^{2n}_{r}(0)\cap \{ q_n = 0 \}) = B^{2n-2}_{r}(0)$, thus
\begin{equation}\label{eq:inc}
\mathcal{U}_{0} \subset T^{*}C\times B^{2n-2}_{r}(0)\times S^{1}.
\end{equation}
Since $\pi r^2\not\in\Z$, by Lemma~\ref{lemma:contgamma},
\begin{equation*}
\gamma \left(T^{*}C\times B^{2n-2}_{r}(0)\times S^{1}\right)
\leq \left\lceil\pi r^2 \right\rceil,
\end{equation*}
thus, Lemma~\ref{lem:capineq}, inclusion (\ref{eq:inc}) and inequality (\ref{eq:applem}) gives
\begin{equation*}
\left\lceil\pi R^2 \right\rceil
\leq
\gamma\left(\mathcal{U}_{0}\right) \leq
\gamma \left(T^{*}C\times B^{2n-2}_{r}(0)\times S^{1}\right)
\leq \left\lceil\pi r^2 \right\rceil,
\end{equation*}
a contradiction with $\pi R^2 > \ell > \pi r^2$.
\end{proof}

\bibliographystyle{amsplain}
\bibliography{biblioCC}

\end{document}

%% file: elementary.pdf_tex
%% Creator: Inkscape inkscape 0.92.2, www.inkscape.org
%% PDF/EPS/PS + LaTeX output extension by Johan Engelen, 2010
%% Accompanies image file 'elementary.pdf' (pdf, eps, ps)
%%
%% To include the image in your LaTeX document, write
%%   \input{<filename>.pdf_tex}
%%  instead of
%%   \includegraphics{<filename>.pdf}
%% To scale the image, write
%%   \def\svgwidth{<desired width>}
%%   \input{<filename>.pdf_tex}
%%  instead of
%%   \includegraphics[width=<desired width>]{<filename>.pdf}
%%
%% Images with a different path to the parent latex file can
%% be accessed with the `import' package (which may need to be
%% installed) using
%%   \usepackage{import}
%% in the preamble, and then including the image with
%%   \import{<path to file>}{<filename>.pdf_tex}
%% Alternatively, one can specify
%%   \graphicspath{{<path to file>/}}
%% 
%% For more information, please see info/svg-inkscape on CTAN:
%%   http://tug.ctan.org/tex-archive/info/svg-inkscape
%%
\begingroup%
  \makeatletter%
  \providecommand\color[2][]{%
    \errmessage{(Inkscape) Color is used for the text in Inkscape, but the package 'color.sty' is not loaded}%
    \renewcommand\color[2][]{}%
  }%
  \providecommand\transparent[1]{%
    \errmessage{(Inkscape) Transparency is used (non-zero) for the text in Inkscape, but the package 'transparent.sty' is not loaded}%
    \renewcommand\transparent[1]{}%
  }%
  \providecommand\rotatebox[2]{#2}%
  \newcommand*\fsize{\dimexpr\f@size pt\relax}%
  \newcommand*\lineheight[1]{\fontsize{\fsize}{#1\fsize}\selectfont}%
  \ifx\svgwidth\undefined%
    \setlength{\unitlength}{490.49860736bp}%
    \ifx\svgscale\undefined%
      \relax%
    \else%
      \setlength{\unitlength}{\unitlength * \real{\svgscale}}%
    \fi%
  \else%
    \setlength{\unitlength}{\svgwidth}%
  \fi%
  \global\let\svgwidth\undefined%
  \global\let\svgscale\undefined%
  \makeatother%
  \begin{picture}(1,1.00611906)%
    \lineheight{1}%
    \setlength\tabcolsep{0pt}%
    \put(0,0){\includegraphics[width=\unitlength,page=1]{elementary.pdf}}%
    \put(0.50852973,0.8913913){\color[rgb]{0,0,0}\makebox(0,0)[lt]{\lineheight{1.25}\smash{\begin{tabular}[t]{l}$b$\end{tabular}}}}%
    \put(0.42600924,0.52126256){\color[rgb]{0,0,0}\makebox(0,0)[lt]{\lineheight{1.25}\smash{\begin{tabular}[t]{l}$x$\end{tabular}}}}%
    \put(0.53886816,0.24943038){\color[rgb]{0,0,0}\makebox(0,0)[lt]{\lineheight{1.25}\smash{\begin{tabular}[t]{l}$x_0$\end{tabular}}}}%
    \put(0.78157552,0.3635028){\color[rgb]{0,0,0}\makebox(0,0)[lt]{\lineheight{1.25}\smash{\begin{tabular}[t]{l}$a$\end{tabular}}}}%
    \put(0.27310362,0.35136746){\color[rgb]{0,0,0}\makebox(0,0)[lt]{\lineheight{1.25}\smash{\begin{tabular}[t]{l}$a'$\end{tabular}}}}%
    \put(0.56313889,0.3695705){\color[rgb]{0,0,0}\makebox(0,0)[lt]{\lineheight{1.25}\smash{\begin{tabular}[t]{l}$\alpha$\end{tabular}}}}%
    \put(0.24033812,0.79188127){\color[rgb]{0,0,0}\makebox(0,0)[lt]{\lineheight{1.25}\smash{\begin{tabular}[t]{l}$\partial D$\end{tabular}}}}%
    \put(-0.00236924,0.12807665){\color[rgb]{0,0,0}\makebox(0,0)[lt]{\lineheight{1.25}\smash{\begin{tabular}[t]{l}$\partial B_{r}^{2n-2}(x_{0})$\end{tabular}}}}%
    \put(0.50003499,0.60985077){\color[rgb]{0,0,0}\makebox(0,0)[lt]{\lineheight{1.25}\smash{\begin{tabular}[t]{l}$c$\end{tabular}}}}%
    \put(0.75002357,0.15477439){\color[rgb]{0,0,0}\makebox(0,0)[lt]{\lineheight{1.25}\smash{\begin{tabular}[t]{l}$\rho(x)$\end{tabular}}}}%
  \end{picture}%
\endgroup%

%% file: function.pdf_tex
%% Creator: Inkscape inkscape 0.92.2, www.inkscape.org
%% PDF/EPS/PS + LaTeX output extension by Johan Engelen, 2010
%% Accompanies image file 'function.pdf' (pdf, eps, ps)
%%
%% To include the image in your LaTeX document, write
%%   \input{<filename>.pdf_tex}
%%  instead of
%%   \includegraphics{<filename>.pdf}
%% To scale the image, write
%%   \def\svgwidth{<desired width>}
%%   \input{<filename>.pdf_tex}
%%  instead of
%%   \includegraphics[width=<desired width>]{<filename>.pdf}
%%
%% Images with a different path to the parent latex file can
%% be accessed with the `import' package (which may need to be
%% installed) using
%%   \usepackage{import}
%% in the preamble, and then including the image with
%%   \import{<path to file>}{<filename>.pdf_tex}
%% Alternatively, one can specify
%%   \graphicspath{{<path to file>/}}
%% 
%% For more information, please see info/svg-inkscape on CTAN:
%%   http://tug.ctan.org/tex-archive/info/svg-inkscape
%%
\begingroup%
  \makeatletter%
  \providecommand\color[2][]{%
    \errmessage{(Inkscape) Color is used for the text in Inkscape, but the package 'color.sty' is not loaded}%
    \renewcommand\color[2][]{}%
  }%
  \providecommand\transparent[1]{%
    \errmessage{(Inkscape) Transparency is used (non-zero) for the text in Inkscape, but the package 'transparent.sty' is not loaded}%
    \renewcommand\transparent[1]{}%
  }%
  \providecommand\rotatebox[2]{#2}%
  \newcommand*\fsize{\dimexpr\f@size pt\relax}%
  \newcommand*\lineheight[1]{\fontsize{\fsize}{#1\fsize}\selectfont}%
  \ifx\svgwidth\undefined%
    \setlength{\unitlength}{571.41598706bp}%
    \ifx\svgscale\undefined%
      \relax%
    \else%
      \setlength{\unitlength}{\unitlength * \real{\svgscale}}%
    \fi%
  \else%
    \setlength{\unitlength}{\svgwidth}%
  \fi%
  \global\let\svgwidth\undefined%
  \global\let\svgscale\undefined%
  \makeatother%
  \begin{picture}(1,0.49281123)%
    \lineheight{1}%
    \setlength\tabcolsep{0pt}%
    \put(0,0){\includegraphics[width=\unitlength,page=1]{function.pdf}}%
    \put(-0.0008716,0.43221618){\color[rgb]{0,0,0}\makebox(0,0)[lt]{\lineheight{1.25}\smash{\begin{tabular}[t]{l}$\pi R^2 + \varepsilon$\end{tabular}}}}%
    \put(0.62515492,0.00360091){\color[rgb]{0,0,0}\makebox(0,0)[lt]{\lineheight{1.25}\smash{\begin{tabular}[t]{l}$4R^2$\end{tabular}}}}%
    \put(0.71428132,0.00509748){\color[rgb]{0,0,0}\makebox(0,0)[lt]{\lineheight{1.25}\smash{\begin{tabular}[t]{l}$4R^2+\delta$\end{tabular}}}}%
    \put(0,0){\includegraphics[width=\unitlength,page=2]{function.pdf}}%
    \put(0.6663596,0.39437036){\color[rgb]{0,0,0}\makebox(0,0)[lt]{\lineheight{1.25}\smash{\begin{tabular}[t]{l}$h+\frac{\varepsilon}{2}$\end{tabular}}}}%
    \put(0.66626256,0.3536421){\color[rgb]{0,0,0}\makebox(0,0)[lt]{\lineheight{1.25}\smash{\begin{tabular}[t]{l}$h_{\varepsilon}$\end{tabular}}}}%
    \put(0.12693675,0.01141611){\color[rgb]{0,0,0}\makebox(0,0)[lt]{\lineheight{1.25}\smash{\begin{tabular}[t]{l}$0$\end{tabular}}}}%
    \put(0,0){\includegraphics[width=\unitlength,page=3]{function.pdf}}%
    \put(0.1036628,0.0495665){\color[rgb]{0,0,0}\makebox(0,0)[lt]{\lineheight{1.25}\smash{\begin{tabular}[t]{l}$\frac{\varepsilon}{2}$\end{tabular}}}}%
    \put(0,0){\includegraphics[width=\unitlength,page=4]{function.pdf}}%
    \put(0.22694067,0.00566557){\color[rgb]{0,0,0}\makebox(0,0)[lt]{\lineheight{1.25}\smash{\begin{tabular}[t]{l}$\delta$\end{tabular}}}}%
    \put(0,0){\includegraphics[width=\unitlength,page=5]{function.pdf}}%
    \put(0.46488812,0.00491731){\color[rgb]{0,0,0}\makebox(0,0)[lt]{\lineheight{1.25}\smash{\begin{tabular}[t]{l}$4R^2-\delta$\end{tabular}}}}%
  \end{picture}%
\endgroup%

%% file: construction.pdf_tex
%% Creator: Inkscape inkscape 0.92.2, www.inkscape.org
%% PDF/EPS/PS + LaTeX output extension by Johan Engelen, 2010
%% Accompanies image file 'construction.pdf' (pdf, eps, ps)
%%
%% To include the image in your LaTeX document, write
%%   \input{<filename>.pdf_tex}
%%  instead of
%%   \includegraphics{<filename>.pdf}
%% To scale the image, write
%%   \def\svgwidth{<desired width>}
%%   \input{<filename>.pdf_tex}
%%  instead of
%%   \includegraphics[width=<desired width>]{<filename>.pdf}
%%
%% Images with a different path to the parent latex file can
%% be accessed with the `import' package (which may need to be
%% installed) using
%%   \usepackage{import}
%% in the preamble, and then including the image with
%%   \import{<path to file>}{<filename>.pdf_tex}
%% Alternatively, one can specify
%%   \graphicspath{{<path to file>/}}
%% 
%% For more information, please see info/svg-inkscape on CTAN:
%%   http://tug.ctan.org/tex-archive/info/svg-inkscape
%%
\begingroup%
  \makeatletter%
  \providecommand\color[2][]{%
    \errmessage{(Inkscape) Color is used for the text in Inkscape, but the package 'color.sty' is not loaded}%
    \renewcommand\color[2][]{}%
  }%
  \providecommand\transparent[1]{%
    \errmessage{(Inkscape) Transparency is used (non-zero) for the text in Inkscape, but the package 'transparent.sty' is not loaded}%
    \renewcommand\transparent[1]{}%
  }%
  \providecommand\rotatebox[2]{#2}%
  \newcommand*\fsize{\dimexpr\f@size pt\relax}%
  \newcommand*\lineheight[1]{\fontsize{\fsize}{#1\fsize}\selectfont}%
  \ifx\svgwidth\undefined%
    \setlength{\unitlength}{580.88379245bp}%
    \ifx\svgscale\undefined%
      \relax%
    \else%
      \setlength{\unitlength}{\unitlength * \real{\svgscale}}%
    \fi%
  \else%
    \setlength{\unitlength}{\svgwidth}%
  \fi%
  \global\let\svgwidth\undefined%
  \global\let\svgscale\undefined%
  \makeatother%
  \begin{picture}(1,0.20452559)%
    \lineheight{1}%
    \setlength\tabcolsep{0pt}%
    \put(0,0){\includegraphics[width=\unitlength,page=1]{construction.pdf}}%
    \put(-0.00128609,0.11311912){\color[rgb]{0,0,0}\makebox(0,0)[lt]{\lineheight{1.25}\smash{\begin{tabular}[t]{l}$B_r\times S^1$\end{tabular}}}}%
    \put(0.18807275,0.00557774){\color[rgb]{0,0,0}\makebox(0,0)[lt]{\lineheight{1.25000012}\smash{\begin{tabular}[t]{l}$P_{R}\times S^1$\end{tabular}}}}%
    \put(0.83015846,0.00162401){\color[rgb]{0,0,0}\makebox(0,0)[lt]{\lineheight{1.25}\smash{\begin{tabular}[t]{l}$P_{R}\times S^1+\frac{\partial}{\partial q_n}$\end{tabular}}}}%
    \put(0.37231655,0.15977318){\color[rgb]{0,0,0}\makebox(0,0)[lt]{\lineheight{1.25}\smash{\begin{tabular}[t]{l}$\tau_{c/4}$\end{tabular}}}}%
    \put(0.58898092,0.05065025){\color[rgb]{0,0,0}\makebox(0,0)[lt]{\lineheight{1.25}\smash{\begin{tabular}[t]{l}$\tau_{3c/4}$\end{tabular}}}}%
    \put(0.22681929,0.10125798){\color[rgb]{0,0,0}\makebox(0,0)[lt]{\lineheight{1.25}\smash{\begin{tabular}[t]{l}$\phi_{1}'$\end{tabular}}}}%
    \put(0.37943325,0.0411613){\color[rgb]{0,0,0}\makebox(0,0)[lt]{\lineheight{1.25}\smash{\begin{tabular}[t]{l}$\widetilde{\phi}_{1}'$\end{tabular}}}}%
  \end{picture}%
\endgroup%

%% file: camel.bbl
\providecommand{\bysame}{\leavevmode\hbox to3em{\hrulefill}\thinspace}
\providecommand{\MR}{\relax\ifhmode\unskip\space\fi MR }
% \MRhref is called by the amsart/book/proc definition of \MR.
\providecommand{\MRhref}[2]{%
  \href{http://www.ams.org/mathscinet-getitem?mr=#1}{#2}
}
\providecommand{\href}[2]{#2}
\begin{thebibliography}{10}

\bibitem{ABKLB}
B.~Aebischer, M.~Borer, M.~K{\"a}lin, C.~Leuenberger, and H.M. Bach,
  \emph{Symplectic geometry: An introduction based on the seminar in bern,
  1992}, Progress in Mathematics, Birkh{\"a}user Basel, 2013.

\bibitem{Bhu}
Mohan Bhupal, \emph{A partial order on the group of contactomorphisms of {$\Bbb
  R^{2n+1}$} via generating functions}, Turkish J. Math. \textbf{25} (2001),
  no.~1, 125--135.

\bibitem{Bus}
Jaime {Bustillo}, \emph{{A coisotropic camel theorem in Symplectic Topology and
  its effect on the Sine-Gordon equation}}, ArXiv e-prints (2017).

\bibitem{Cha95}
Marc Chaperon, \emph{On generating families}, The {F}loer memorial volume,
  Progr. Math., vol. 133, Birkh\"auser, Basel, 1995, pp.~283--296.

\bibitem{Che96}
Yuri~Vitalievich Chekanov, \emph{Critical points of quasifunctions, and
  generating families of {L}egendrian manifolds}, Funktsional. Anal. i
  Prilozhen. \textbf{30} (1996), no.~2, 56--69, 96.

\bibitem{Chi}
Sheng-Fu Chiu, \emph{Nonsqueezing property of contact balls}, Duke Math. J.
  \textbf{166} (2017), no.~4, 605--655.

\bibitem{EH1}
I.~Ekeland and H.~Hofer, \emph{Symplectic topology and {H}amiltonian dynamics},
  Math. Z. \textbf{200} (1989), no.~3, 355--378.

\bibitem{EH2}
Ivar Ekeland and Helmut Hofer, \emph{Symplectic topology and {H}amiltonian
  dynamics. {II}}, Math. Z. \textbf{203} (1990), no.~4, 553--567.

\bibitem{EG}
Yakov Eliashberg and Mikhael Gromov, \emph{Convex symplectic manifolds},
  Several complex variables and complex geometry, {P}art 2 ({S}anta {C}ruz,
  {CA}, 1989), Proc. Sympos. Pure Math., vol.~52, Amer. Math. Soc., Providence,
  RI, 1991, pp.~135--162.

\bibitem{EKP}
Yakov Eliashberg, Sang~Seon Kim, and Leonid Polterovich, \emph{Geometry of
  contact transformations and domains: orderability versus squeezing}, Geom.
  Topol. \textbf{10} (2006), 1635--1747.

\bibitem{FHV}
A.~Floer, H.~Hofer, and C.~Viterbo, \emph{The {W}einstein conjecture in
  {$P\times {\bf C}^l$}}, Math. Z. \textbf{203} (1990), no.~3, 469--482.

\bibitem{Fra16}
Maia Fraser, \emph{Contact non-squeezing at large scale in {$\Bbb R^{2n}\times
  S^1$}}, Internat. J. Math. \textbf{27} (2016), no.~13, 1650107, 25.

\bibitem{Gro85}
Mikhael Gromov, \emph{Pseudo holomorphic curves in symplectic manifolds},
  Invent. Math. \textbf{82} (1985), no.~2, 307--347.

\bibitem{HZ}
H.~Hofer and E.~Zehnder, \emph{Periodic solutions on hypersurfaces and a result
  by {C}. {V}iterbo}, Invent. Math. \textbf{90} (1987), no.~1, 1--9.

\bibitem{San}
Sheila Sandon, \emph{Contact homology, capacity and non-squeezing in {$\Bbb
  R^{2n}\times S^1$} via generating functions}, Ann. Inst. Fourier (Grenoble)
  \textbf{61} (2011), no.~1, 145--185.

\bibitem{Sik87}
Jean-Claude Sikorav, \emph{Probl\`emes d'intersections et de points fixes en
  g\'eom\'etrie hamiltonienne}, Comment. Math. Helv. \textbf{62} (1987), no.~1,
  62--73.

\bibitem{Sik90}
\bysame, \emph{Systèmes hamiltoniens et topologie symplectique}, 10 1990.

\bibitem{ThePHD}
David Theret, \emph{Utilisation des fonctions generatrices en geometrie
  symplectique globale}, Ph.D. thesis, 1996, Thèse de doctorat dirigée par
  Chaperon, Marc Mathématiques Paris 7 1996, p.~1 vol. (117 P.).

\bibitem{The}
\bysame, \emph{A complete proof of {V}iterbo's uniqueness theorem on generating
  functions}, Topology Appl. \textbf{96} (1999), no.~3, 249--266.

\bibitem{Tray}
Lisa Traynor, \emph{Symplectic homology via generating functions}, Geom. Funct.
  Anal. \textbf{4} (1994), no.~6, 718--748.

\bibitem{Vit}
Claude Viterbo, \emph{Symplectic topology as the geometry of generating
  functions}, Math. Ann. \textbf{292} (1992), no.~4, 685--710.

\end{thebibliography}
